\documentclass[pdflatex,12pt]{article}
\usepackage[utf8]{inputenc}

\usepackage{amsmath,color}
\usepackage{amssymb,amsthm}

\usepackage[permil]{overpic}
\usepackage[multiple]{footmisc}
\usepackage{xr}
\usepackage[left=1.6cm,right=1.6cm,top=2.50cm,bottom=2.50cm]{geometry}
\usepackage{graphicx}
\usepackage[font=small,labelfont=bf,
   justification=justified,
   format=plain,labelsep=space]{caption}
\usepackage{subfig}
\usepackage{float}
\usepackage[english]{babel}
\usepackage[font=small,labelfont=bf,justification=centering]{caption}
\usepackage[T1]{fontenc}
\usepackage{lastpage}
\usepackage{enumerate}
\usepackage{enumitem}
\usepackage{lmodern}
\usepackage{array}
\usepackage{bm}
\usepackage{multirow}
\usepackage{dsfont}
\usepackage{tensor}
\usepackage{fancyhdr}
\usepackage{listings}
\usepackage{dsfont}
\usepackage{siunitx}
\usepackage{titling}
\usepackage{lipsum}
\usepackage{tabularx}
\usepackage{verbatim}
\usepackage{authblk}
\usepackage{csquotes}
\usepackage{bbm} %for indicator function
\usepackage{hyperref} % good for references
\usepackage{ctable} % for \specialrule command for tables
\usepackage{multirow} % Merge rows in tables 
\usepackage{titlesec} % Modify font sections, subsections, etc.
\usepackage{marvosym}
\usepackage{relsize,exscale} % for big integrals
\usepackage[
backend=biber,
style=numeric,
giveninits=true,
sorting=anyt
]{biblatex}

\allowdisplaybreaks

\graphicspath{{Fig/}}

\addto\captionsenglish{}

\titleformat{\subsection}{\normalfont\large\raggedright\it}{\thesubsection}{1em}{}

\makeatletter
\newcommand{\thickhline}{%
    \noalign {\ifnum 0=`}\fi \hrule height 1pt
    \futurelet \reserved@a \@xhline
}
\newcolumntype{"}{@{\hskip\tabcolsep\vrule width 1pt\hskip\tabcolsep}}
\makeatother

\usepackage{adjustbox}

\makeatletter
\newsavebox{\@brx}
\newcommand{\llangle}[1][]{\savebox{\@brx}{\(\m@th{#1\langle}\)}%
  \mathopen{\copy\@brx\kern-0.5\wd\@brx\usebox{\@brx}}}
\newcommand{\rrangle}[1][]{\savebox{\@brx}{\(\m@th{#1\rangle}\)}%
  \mathclose{\copy\@brx\kern-0.5\wd\@brx\usebox{\@brx}}}
\makeatother

\newcommand{\RomanNumeralCaps}[1]
    {\MakeUppercase{\romannumeral #1}}

\newtheorem{deff}{Definition}[section]
\newtheorem{prop}[deff]{Proposition}
\newtheorem{thm}[deff]{Theorem}
\newtheorem{lm}[deff]{Lemma}

\theoremstyle{definition}
\newtheorem{example}[deff]{Example}
\newtheorem{remark}[deff]{Remark}

\title{Early-Warning Signs for SPDEs with Continuous Spectrum}

\author[*,$\ddag$]{P. Bernuzzi}
\author[*]{A. D\"ux}
\author[*,$\dag$]{C. Kuehn}
\affil[*]{\footnotesize{Technical University of Munich, School of Computation Information and Technology, Department of
Mathematics, Boltzmannstraße 3, 85748 Garching, Germany}}
\affil[$\dag$]{\footnotesize{Technical University of Munich, Munich Data Science Institute, Walther-von-Dyck-Straße 10, 85748 Garching, Germany}}
\affil[$\ddag$]{Author to whom any correspondence should be addressed. Email: paolo.bernuzzi@ma.tum.de}

\date{\today}

\begin{document}

\maketitle{}

\begin{abstract}
%Fast-slow dynamical systems are useful tools to describe various real life scenarios. The presence of slow variables with strong influence on fast ones makes the study of early-warning signs an important field. 
In this work, we study early-warning signs for stochastic partial differential equations (SPDEs), where the linearization around a steady state has continuous spectrum. The studied warning sign takes the form of qualitative changes in the variance as a deterministic bifurcation threshold is approached via parameter variation. Specifically, we focus on the scaling law of the variance near the transition. Since we are dealing here, in contrast to previous studies, with the case of continuous spectrum and quantitative scaling laws, it is natural to start with linearizations that are multiplication operators defined by analytic functions. For a one-dimensional spatial domain we obtain precise rates of divergence.  In the case of the two- and three-dimensional domains an upper bound to the rate of the early-warning sign is proven. These results are cross-validated by numerical simulations. Our theory can be generically useful for several applications, where stochastic and spatial aspects are important in combination with continuous spectrum bifurcations.
\end{abstract}


\small{\textbf{Acknowledgments:} This project has received funding from the European Union’s Horizon 2020 research and innovation programme under Grant Agreement 956170.}

% Header
\pagestyle{fancy}
\fancyhead{}
\renewcommand{\headrulewidth}{0pt}
\fancyhead[C]{\textit{Early-Warning Signs for SPDEs with Continuous Spectrum}}

\section{Introduction}

One natural way, how critical transitions appear in complex systems is in the context of  differential equations with multiple time scales~\cite{kuehn2015multiple}. After a longer period of slow change, a critical transition corresponds to a large/drastic event happening on a fast scale. In more detail, we can characterize this mechanism \cite{KU4,KU} by several essential components. The typical dynamics of the system for most times constitutes a slow motion. Yet, there is generically a slow drift towards a bifurcation point of the fast/layer subsystem. The critical transition is an abrupt change that happens on a fast timescale after the bifurcation point has been passed. Of course, it is of interest to study, whether there are early-warning signs to potentially anticipate a critical transition. One crucial mechanism to extract warning signs from data is to exploit the effect of critical slowing down, where the recovery of perturbations to the current state starts to slow as a bifurcation is approached. This effect can then be extracted from data by exploiting natural stochastic fluctuations as perturbation effects and by measuring a stochastic observable such as variance or autocorrelation~\cite{Wiesenfeld1}. This approach has recently gained considerable popularity in various applications \cite{SCHE}.
\\

It is evident that in many complex systems with critical transitions, we should also take into account, beyond stochasticity, a spatial component. A prime example are applications in neuroscience, where it is evident that larger-scale events, such as epileptic seizures, are spatial stochastic systems with critical transitions~\cite{MCS,ME,MO}. Another example are electric power systems where critical transitions can lead to a failure of the whole system such as voltage collapses~\cite{COT}. Other examples of complex systems with critical transitions occur in medical applications, including asthma attacks~\cite{VE} and epidemic outbreaks~\cite{ORE}, economic applications, like financial crisis~\cite{MA}, and environmental applications, such as climate changes~\cite{AL,LE} and population extinctions~\cite{SCHE1}. It becomes apparent that for each of these applications, a prediction and warning of a critical transition is desirable as it could enable the prevention of its occurrence or at least improve the adaptability to it. Yet, to make early-warning sign theory reliable and avoid potential non-robust conclusions~\cite{Dakosetal,DitlevsenJohnsen}, a rigorous mathematical study of early-warning sign theory is needed. There exists already a quite well-developed theory for warning signs of systems modeled  by stochastic ordinary differential equations (SODE), for example in \cite{CR,KU6}. However, there are still many open problems for systems modelled by stochastic partial differential equations (SPDE). For SPDEs is potentially even more crucial in comparison to SODEs to develop a mathematical theory as SPDEs model complex systems, where it is either extremely costly to obtain high-resolution experimental, or even simulation, data in many applications. Therefore, experimentally-driven approaches have to be guided by rigorous mathematical analysis.\\

For our analysis we follow the construction in \cite{KU}, which considers differential equations of the form
\begin{align} \label{eq:fastslow}
\begin{cases}
    \text{d}u(x,t)=F_1(x)u(x,t)+F_2(u(x,t),p(x,t))\, \text{d}t + \sigma \, \text{d}W(t) \, , \\ 
    \text{d}p(x,t)= \epsilon G(u(x,t),p(x,t))\, \text{d}t \, ,
\end{cases}
\end{align}
with $(x,t) \in \mathcal{X} \times [0, \infty)$ for a connected set $\mathcal{X}\subseteq\mathbb{R}^N$ with nonempty interior, $u:\mathcal{X} \times [0, \infty)\to \mathbb{R}$, $p:\mathcal{X} \times [0, \infty)\to \mathbb{R}$ and $N\in\mathbb{N}:=\{1,2,3,\ldots\}$. Suppose that $F_1$ is a linear operator and that $F_2$ and $G$ are sufficiently smooth nonlinear maps. We assume $\sigma>0$, $0<\epsilon\ll 1$ and with the notation $W$ we refer to a $Q$-Wiener process. The precise properties of the operator $Q$ are described further below. We denote $u(x,t)$ by $u$ and $p(x,t)$ as $p$.  \\

Suppose that $F_2(0,p) = 0$. Hence, for any $p$ we obtain the homogeneous steady state $u_{\ast}\equiv 0$ for the deterministic partial differential equation (PDE) corresponding to \eqref{eq:fastslow} with $\sigma=0$. To determine the local stability of $u_{\ast} \equiv 0$, we study the linear operator
\begin{align*}
    A(p) := F_1 + \text{D}_{u_{\ast}} F_2(0,p) \, ,
\end{align*}
where $\text{D}_u$ is the Fréchet derivative of $u$ on a suitable Banach space, which contains the solutions of \eqref{eq:fastslow} and which is described in full detail below.\\

Suppose that for $p <0$ the spectrum of $A(p)$ is contained in $\{ z: \text{Re}(z) < 0\}$, whereas for $p>0$ the spectrum has elements in $\{ z: \text{Re}(z) > 0\}$. Therefore, the fast subsystem, given by taking the limit $\epsilon=0$, has a bifurcation point at $p=0$. For the full fast-slow system~\eqref{eq:fastslow} with $\epsilon>0$, the slow dynamics $\partial_t p = \epsilon G(0,p)$ can drive the system to the bifurcation point at $p=0$ and potentially induce a critical transition. For the case $\epsilon=0$, the variable $p$ just becomes a parameter and one may still study the motion of $u$ under variation of $p$. More precisely, as $p$ is varying, we can cross the bifurcation threshold $p=0$ present in the fast PDE dynamics $\partial_t u = F_1 u + F_2(u,p)$. For simplicity, we just consider here the case $\epsilon = 0$ and that $p$ is a parameter but see~\cite{GnannKuehnPein} for steps towards a more abstract theory of fast-slow SPDEs. The problem studied is hence of the form of the linearized fast system
\begin{align} \label{eq:originalform}
    \text{d}U(x,t) = A(p) U(x,t) \, \text{d}t + \sigma \, \text{d}W(t) \, .
\end{align}
We discuss this problem setting initially $A(p) = p \operatorname{I} + \operatorname{T}_f$, where $\operatorname{I}$ denotes the identity and $\operatorname{T}_f$ is a multiplication operator associated to a function $f: \mathcal{X} \to \mathbb{R}$. Such choice is directly motivated by the spectral theorem \cite{hall2013quantum}, which associates a multiplication operator to other operators with same spectrum. For the precise technical set-up and assumptions made, we refer to Section \ref{sec:preliminaries}.\\

As mentioned above, prediction of critical transitions is very important in applications as it can help reduce, or even completely evade, its effects. This can be done by finding early-warning signs~\cite{bernuzzi2023bifurcations,blumenthal2021pitchfork}. Our goal is to analyze the problem for early-warning signs in form of an increasing variance as the bifurcation point at $p=0$ is approached. \\

This paper is structured as follows. In Section \ref{sec:preliminaries} we set up the problem and state our main results. In Section \ref{sec:1-dim} we discuss the variance of the solutions of \eqref{eq:originalform} for measurable functions $f$ that are defined on a one-dimensional space, i.e., assuming that $N =1$. We first focus on a specific tool function type and then generalize the obtained results to the case of more general analytic functions. In Section \ref{sec:N-dim} we proceed with the study of the scalar product that defines the variance along certain directions in the space of square-integrable functions, with the domain of $f$ in higher dimensions. While observing convergence as $p\to 0^-$ for $N>1$ dimensions and $f$ under certain assumptions, we further find an upper bound of divergence for analytic functions $f$ on two- and three-dimensional domains. In Section \ref{sec:gen} we discuss how our results are affected by relaxing certain assumptions and we provide examples of the effectiveness of the early warning-signs for different types of linear operators presented in the drift term of \eqref{eq:originalform}. Lastly, in Section \ref{sec:num} we approximate \eqref{eq:originalform} for one- and two-dimensional spatial domains to obtain visual numerical representations and a cross-validation of the discussed theorems.

\section{Preliminaries} \label{sec:preliminaries}
In this paper we study the stochastic partial differential equation
\begin{align} \label{sist}
    \begin{cases}
    \text{d}u(x,t) = \left(f(x)+p\right)\; u(x,t) \, \text{d}t + \sigma \text{d}W(t)\\
    u(\cdot,0) = u_0
    \end{cases}
\end{align}
for $x\in\mathcal{X}$, $t>0$ and $\mathcal{X}\subseteq\mathbb{R}^N$ a connected set with non-empty interior for $N\in\mathbb{N}$. The operator $\operatorname{T}_f$ denotes the realization in $L^2(\mathcal{X})$ of the multiplication operator for $f$, i.e.
\begin{align*}
    \operatorname{T}_fu=fu\quad,\quad \mathcal{D}(\operatorname{T}_f):=\left\{ u\;\;\text{Lebesgue measurable}\;\Big|\; fu\in L^2(\mathcal{X}) \right\}\quad,
\end{align*}
with $\mathcal{D}(\cdot)$ denoting the domain of an operator. We assume that the initial condition of the system satisfies $u_0\in \mathcal{D}(\operatorname{T}_f)$ and $L^2(\mathcal{X})$ is the Hilbert space of square-integrable functions with the usual the scalar product $\langle\cdot,\cdot\rangle$. We denote as $H^{s}(\mathcal{X})$ the Hilbert Sobolev space of degree $s$ on $\mathcal{X}$. We are going to assume that the Lebesgue measurable function $f:\mathcal{X}\to\mathbb{R}$ satisfies
\begin{align} \label{property_f}
    f(x)<0~~ \forall x\in\mathcal{X}\setminus\{x_\ast\}, \quad \int_{\mathcal{X}} \frac{1}{|f(x)-p|} \text{d}x<+\infty \text{  for any $p<0$, and} \quad 
    f(x_\ast)=0,
\end{align}
for a point $x_\ast\in\mathcal{X}$. Without loss of generality, i.e., up to a translation coordinate change in space, we can assume $x_\ast=0$, where $0$ denotes as the origin in $\mathbb{R}^N$. We can then assume, again using a translation of space, that there exists $\delta>0$ such that $[0,\delta]^N\subseteq \mathcal{X}$. We assume also $\sigma>0$ and $p<0$.\\
The construction of the noise process $W$ and the existence of the solution follow from standard methods~\cite{DP} as follows. For fixed $T>0$, we consider the probability space $\left(\Omega, \mathcal{F}, \mathbb{P}\right)$ with $\Omega$ the space of paths $\{X(t)\}_{t\in [0,T]}$ such that $X(t)\in L^2(\mathcal{X})$ for any $t\in[0,T]$, for $\mathcal{F}$ the Borel $\sigma$-algebra on $\Omega$ and with $\mathbb{P}$ a Wiener measure induced by a Wiener process on $[0,T]$. To this probability space we associate the natural filtration $\mathcal{F}_t$ of $\mathcal{F}$ with respect to a Wiener process for $t\in[0,T]$.\\ 
We consider the non-negative operator $Q:L^2(\mathcal{X})\to L^2(\mathcal{X})$ with eigenvalues $\{q_k\}_{k\in\mathbb{N}}$ and eigenfunctions $\{b_k\}_{k\in\mathbb{N}}$.
For $Q$ assumed to be trace-class, the stochastic process $W$ is called $Q$-Wiener process if
\begin{align} \label{QW}
    W(t)=\sum_{i=1}^\infty \sqrt{q_i} b_i \beta^i(t)
\end{align}
holds for a family $\{\beta^i\}_{i\in\mathbb{N}}$ of independent, identically distributed and $\mathcal{F}_t$-adapted Brownian motions. Such series converges in $L^2(\Omega,\mathcal{F},\mathbb{P})$ by \cite[Proposition 4.1]{DP}. Throughout the paper we assume $Q$ to be a bounded operator. In this case there exists a larger space (\cite[Propositon 4.11]{DP}), denoted by $H_1$, on which $Q^\frac{1}{2} L^2(\mathcal{X})$ embeds through a Hilbert-Schmidt operator and such that the series in \eqref{QW} converges in $L^2(\Omega,\mathcal{F},\mathbb{P};H_1)$. The resulting stochastic process $W$ is called a generalized Wiener process.\\
Assuming $u_0$ to be $\mathcal{F}_0$-measurable, properties \eqref{property_f} imply that there exists a unique mild solution for the system \eqref{sist} of the form
\begin{align*}
    u(\cdot,t)=e^{\left(f(\cdot)+p\right)t}u_0+ \sigma \int_0^t e^{\left(f(\cdot)+p\right)(t-s)} \; \textnormal{d}W(s)\quad,
\end{align*}
with $t>0$, as stated by \cite[Theorem 5.4]{DP}. From \cite[Theorem 5.2]{DP} one can obtain the covariance operator of the mild solution of \eqref{sist} as
\begin{align*}
    V(t):=\sigma^2 \int_0^t e^{\operatorname{T}_{f+p} s} Q e^{\operatorname{T}_{f+p} s} \;\textnormal{d}s
\end{align*}
for any $t>0$. We also define $V_\infty:=\underset{t\to\infty}{\lim} V(t)$, which is the operator that is used to construct the early-warning signs that we are interested in, i.e., we are interested in the scaling laws that the covariance operator can exhibit as we vary the parameter $p$ towards the bifurcation point at $p=0$.
\\
We employ the big theta notation \cite{knuth1976big}, $\Theta$, in the limit to $0$, i.e. $b_1=\Theta(b_2)$ if $\underset{s\to 0}{\lim} \frac{b_1(s)}{b_2(s)},\;\underset{s\to 0}{\lim} \frac{b_2(s)}{b_1(s)}>0$ for $b_1,b_2$ locally positive functions. The direction along which such limit is approached is indicated by the space on which the parameter is defined.

\subsection{Results}
We primarily set $Q=\operatorname{I}$, the identity operator on $L^2(\mathcal{X})$. The definition of $V_\infty$ implies, through Fubini's Theorem, that
\begin{align}
    \label{eq:cov_integral}
    \langle V_{\infty}g_1, g_2\rangle 
    &= \int_{\mathcal{X}} g_1(x) \sigma^2 \lim_{t \to \infty} \int_0^t e^{(f(x)+p)r} Q e^{(f(x)+p)r} g_2(x) \, \text{d}r \, \text{d}x \\
    & = \int_{\mathcal{X}} g_1(x)g_2(x) \frac{-\sigma^2}{2(f(x)+p)} \, \text{d}x 
    =\dfrac{\sigma^2}{2} \bigg\langle  \frac{-1}{f+p} g_1, g_2 \bigg\rangle\, , \nonumber
\end{align}
with $g_1, g_2 \in L^2(\mathcal{X})$. We set $g_1=g_2=g$ and, initially, we take  $g=\mathbbm{1}_{[0,\varepsilon]^N}$ or $g=\mathbbm{1}_{[-\varepsilon,\varepsilon]^N}$ for small enough $\varepsilon>0$. These indicator functions allow us to understand the scaling laws near the bifurcation point analytically as they are localized near the zero of $f$ at $x_\ast=0$. More general classes of test functions for the inner product with $V_\infty$ can then be treated in the usual way via the approximation with simple functions.\\

In Section \ref{sec:1-dim}, Theorem \ref{thm:1-dim_analytic} provides the scaling law of $\langle g, V_\infty g \rangle$ for $N=1$ and for an analytic function $f:\mathcal{X}\to\mathbb{R}$. Similarly, in Section \ref{sec:N-dim}, Theorem \ref{thm:2-dim_analytic} and Theorem \ref{thm:3-dim_analytic} provide an upper bound for the rate of divergence of $\langle g, V_\infty g \rangle$ for $N=2$ and $N=3$ respectively. Further, similar results are obtained under relaxed assumptions in Section \ref{sec:gen} on $g$, the operator $Q$ and the linear operator present in the drift component of \eqref{sist}. Lastly, Section \ref{sec:num} provides plots that cross-validate the stated conclusions through the use of implicit Euler-Maruyama method.

\section{One-dimensional case} \label{sec:1-dim}
In the current section we assume $N=1$ and obtain precise rate of divergence for the variance of the solution of \eqref{sist} for different types of functions $f: \mathcal{X} \subseteq \mathbb{R} \to \mathbb{R}$. Such behaviour defines an early-warning sign that preceeds a bifurcation threshold of the system as $p$ approaches $0$ from below.\\
First, we choose a specific function type $f$ to analyze. Afterwards, we expand our analysis by considering general analytic functions and, lastly, we discuss an example in which $f$ is not analytic.

\subsection{Tool function} \label{ch:exfctonedim}
Consider the function $f_{\alpha}$ defined by
\begin{align} \label{eq:exfunc}
    f_{\alpha}(x) := - |x|^{\alpha} \quad \text{ with } \alpha > 0 
\end{align}
with $x$ in a neighbourhood of $x_\ast=0$ and such that \eqref{property_f} holds. This type of function is a useful tool for the study of system \eqref{sist}. Indeed, if for $f$ there exist $c_1,c_2,\varepsilon>0$ such that
\begin{align} \label{hyperb_conv}
    c_1 f_\alpha(x)\leq f(x)\leq c_2 f_\alpha(x)\quad\quad\quad \forall x \in [-\varepsilon,\varepsilon]\cap\mathcal{X},
\end{align}
then we may transfer scaling laws obtained from the tool function directly to results for $f$. We note that up to rescaling of the spatial variable we can assume $\varepsilon=1$ and that $[0,1]\subseteq\mathcal{X}$. Hence, using $\operatorname{T}_{f_{\alpha}}$ to analyse the covariance operator and by \eqref{eq:cov_integral} we obtain 
\begin{align*} 
    \langle V_{\infty}g, g \rangle 
    = - \bigg\langle  \frac{\sigma^{2}}{2(f(x)+p)}g, g \bigg\rangle 
    = \int_{0}^{1} \frac{\sigma^{2}}{
    2(-f(x)-p)} \, \text{d}x 
 \end{align*}
and
\begin{align}  \label{eq:dim1_tool}
    \frac{1}{2c_1}\int_{0}^{1} \frac{\sigma^{2}}{
    x^{\alpha}-p} \, \text{d}x 
    \leq \int_{0}^{1} \frac{\sigma^{2}}{
    2(-f(x)-p)} \, \text{d}x   
    \leq \frac{1}{2c_2}\int_{0}^{1} \frac{\sigma^{2}}{
    x^{\alpha}-p} \, \text{d}x
 \end{align}
for $g(x)=\mathbbm{1}_{[0,1]}$ and certain constants $c_1\geq1\geq c_2$.

\begin{figure}[h!]
    \centering   
    \begin{overpic}[width= 0.6\textwidth]{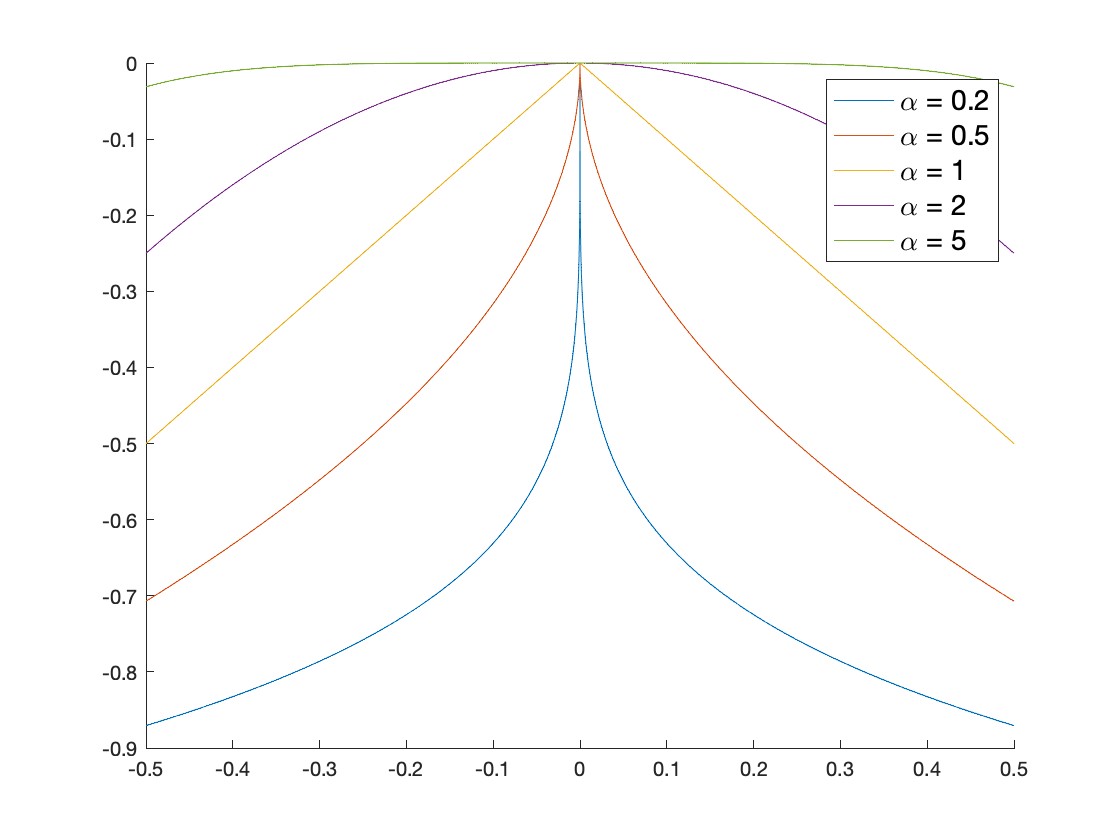}
    \put(510,30){\footnotesize{$x$}}
    \put(40,470){\rotatebox{270}{\footnotesize{$-|x|^{\alpha}$}}}
    \end{overpic}
    \caption{Plot of function $-|x|^{\alpha}$ for different choices of $\alpha$. For $\alpha>1$ the function is $C^1$, with derivative equal to $0$ at $x=0$, therefore flat in $0$. Conversely, for $\alpha\leq 1$ the function is steep at $x=0$.}
    \label{fig:modfunc}
\end{figure}

The following theorem describes the rate of divergence assumed by $\langle g, V_\infty g\rangle$ as $p\to0^-$ for $f=f_\alpha$.

\begin{thm} \label{thm:1-dim_tool}
    For $f=f_\alpha$  defined in \eqref{eq:exfunc}, $Q=\operatorname{I}$ and $\varepsilon>0$, the time-asymptotic covariance $V_\infty$ of the solution of \eqref{sist} along $g(x)=\mathbbm{1}_{[0,\varepsilon]}$, 
    \begin{equation*}
    \langle V_{\infty}g, g \rangle, 
    \end{equation*}
    has a scaling law as $p\to 0^-$ described in Table \ref{tab:onedim}.
\end{thm}

\def\arraystretch{1.5}\tabcolsep=10pt
\begin{table}[h!]
  \begin{center}
    \begin{tabular}{c|c}
      \specialrule{.1em}{.05em}{.05em}
      \textbf{Case} & \textbf{Scaling law for $p \to 0^{-}$} \\
      \hline
      $0< \alpha < 1$ & $ 1 $ \\
      $\alpha = 1$ & $- \log(-p)$ \\
      $\alpha > 1$ & $(-p)^{-1 + \frac{1}{\alpha}}$ \\
      $\alpha \to \infty$ & $(-p)^{-1}$\\
      \specialrule{.1em}{.05em}{.05em}
    \end{tabular}
    \caption{Scaling law in dimension $N=1$ for the function $f=f_{\alpha}(x)$ and $g=\mathbbm{1}_{[0,\varepsilon]}$. }
    \label{tab:onedim}
  \end{center}
\end{table}

\begin{proof}
The scaling law of $\langle V_{\infty}g, g \rangle$ for $p\to 0^-$ is equivalent to the one exhibited by $\int_{0}^{1} \frac{1}{x^{\alpha}-p} \, \text{d}x $, as described in \eqref{eq:dim1_tool}. In such limit in $p$ we obtain
\begin{align*}
    \lim_{p \to 0^{-}} \int_{0}^{1} \frac{1}{x^{\alpha}-p} \, \text{d}x 
    = \int_{0}^{1} \frac{1}{x^{\alpha}} \, \text{d}x = 
    \begin{cases}
    < \infty &\text{for } 0 < \alpha < 1 \, , \\
    = \infty &\text{for } \alpha \geq 1 \, . 
    \end{cases}
\end{align*}
For the case $\alpha = 1$, by substituting $y = x - p$ we find that
\begin{align*} 
    \int_{0}^{1} \frac{1}{x - p} \, \text{d}x = \int_{-p}^{1 - p} \frac{1}{y} \, \text{d}y = \log (1 - p) - \log(-p) = \Theta(-\log(-p)) \quad \text{ for } p \to 0^{-} \, . 
 \end{align*}

We now consider the case $\alpha > 1$. Through the substitution $y=x(-p)^{-\frac{1}{\alpha}}$ we obtain
\begin{align} \label{rate}
    \int_{0}^{1} \frac{1}{x^{\alpha}-p} \, \text{d}x = \frac{1}{-p}\mathlarger{\int}_{0}^{1} \frac{1}{\left(x(-p)^{-\frac{1}{\alpha}}\right)^{\alpha}+1}\, \text{d}x = (-p)^{-1 + \frac{1}{\alpha}} \int_{0}^{(-p)^{-\frac{1}{\alpha}}} \frac{1}{y^{\alpha}+1}\, \text{d}y \, .
\end{align}
Since $\alpha > 1$, we find that 
\begin{align}  \label{eq:onedim_conv}
    \lim_{p \to 0^{-}} \int_{0}^{ (-p)^{-\frac{1}{\alpha}}} \frac{1}{y^{\alpha}+1}\, \text{d}y < \infty \, . 
 \end{align}
Hence, taking the limit in $p$ to zero from below for the equation \eqref{rate} we find divergence with the rate
\begin{align*} 
    (-p)^{-1 + \frac{1}{\alpha}} \int_{0}^{ (-p)^{-\frac{1}{\alpha}}} \frac{1}{y^{\alpha}+1}\, \text{d}y = \Theta\left((-p)^{-1 +\frac{1}{\alpha}}\right)\, .
 \end{align*}
Further, we see that for $\alpha$ approaching infinity this expression is also well-defined with rate of divergence $\Theta\left(\frac{1}{-p}\right)$.
\end{proof}

Theorem \ref{thm:1-dim_tool} states that for the function $f_{\alpha}$ defined by (\ref{eq:exfunc}) the variance along $g$ is converging for $0 < \alpha < 1$. Such function types $f_{\alpha}$ display a pointy shape on $0$, as seen in Figure \ref{fig:modfunc}, given by the fact that their left and right first derivative assume infinite values. Alternatively, we observe that for $\alpha \geq 1$ the variance along $g$ is diverging, indicating that as the graph gets flatter and smoother increasing $\alpha$ a divergence appears. It is interesting to note that a similar scaling law behaviour, associated to the intermittency scaling law of an ODE dependent on a parameter near a non-smooth fold bifurcation, has been found in \cite[Table 1]{KU3}.

\subsection{General analytic functions}
We have found that $\langle g,V_\infty g \rangle$ converges for $\alpha < 1$ and diverges for $\alpha \geq 1$, with corresponding rate of divergence, for the limit $p\to0^-$ with the functions $f_{\alpha}(x) = - |x|^{\alpha}$ and $g=\mathbbm{1}_{[0,\varepsilon]}$. Our aim in this subsection is to generalize this result considerably. In particular, we consider now analytic functions $f$ such that \eqref{property_f} holds.
Applying Taylor's theorem and due to the fact that $f$ vanishes at $x_\ast=0$, the function $f$ is of the form 
\begin{align*}
    f(x)= -\sum_{n=1}^{\infty}a_n x^{n}
\end{align*} 
for any $x\in\mathcal{X}$ and with coefficients $\{a_n\}_{n\in\mathbb{N}}\subset\mathbb{R}$ such that \eqref{property_f} holds. We assume, up to reparametrization, that $[0,\varepsilon]\subseteq\mathcal{X}$. The following theorem provides scaling law, which can be used as an early-warning sign, of the expression $\langle g, V_\infty g\rangle$ for an analytic $f$ and $g(x)=\mathbbm{1}_{[0,\varepsilon]}$.

\begin{thm} \label{thm:1-dim_analytic}
    Set $f(x)=-\overset{\infty}{\underset{n=1}{\sum}}a_n x^{n}$ for all $x \in \mathcal{X} \subseteq \mathbb{R}$, that satisfies \eqref{property_f},  $\{a_n\}_{n\in\mathbb{N}}\subset\mathbb{R}$ and $Q=\operatorname{I}$. Let $m\in\mathbb{N}$ denote the index for which $a_n = 0$ for any $n\in\{1,\dots,m-1\}$ and $a_m\neq0$.\footnote{For odd $m$ it is implied that $\mathcal{X}\subseteq\mathbb{R}_{\geq0}:= \{ x\in \mathbb{R}: x \geq 0 \}$, or $\mathcal{X}\subseteq\mathbb{R}_{\leq0}:= \{ x\in \mathbb{R}: x \leq 0 \}$ up to rescaling, due to the sign of $f$.} Then the time-asymptotic covariance of the solution of \eqref{sist} along $g(x)=\mathbbm{1}_{[0,\varepsilon]}$, $\langle V_{\infty}g, g \rangle$, has the scaling laws as $p\to 0^-$ as described in Table \ref{tab:onedim}, now depending on the value of $m=\alpha>0$.\footnote{Alternatively, we can assume to have a function $f \in C^k$ such that $k\geq m$. This leads to the study of the Peano remainder instead of the whole series.} 
\end{thm}

\begin{proof}
As in the proof of Theorem \ref{thm:1-dim_tool}, up to rescaling of the variable $x$, we can choose $\varepsilon=1$.
Analyzing the variance along $g$, we obtain that
\begin{align}  \label{eq:onedim_analytic_a1}
     \langle V_{\infty}g, g \rangle  &= \frac{\sigma^2}{2} \mathlarger{\mathlarger{\int}}_{0}^{1} \frac{1}{ \overset{\infty}{\underset{n=m}{\sum}}a_nx^n-p }\, \text{d}x  \quad .
 \end{align}
For positive $x$ close to zero, the sum $\overset{\infty}{\underset{n=1}{\sum}}a_nx^n$ is dominated by the leading term $a_m x^m$, since
\begin{align*} 
    \lim_{x \to 0^{+}} \frac{\overset{\infty}{\underset{n=1}{\sum}}a_nx^n}{a_m x^m} 
    = \lim_{x \to 0^{+}} \frac{\overset{\infty}{\underset{n=m}{\sum}}a_nx^n}{a_m x^m} 
    = \lim_{x \to 0^{+}} \frac{a_m + \overset{\infty}{\underset{n=m+1}{\sum}}a_nx^{n-m}}{a_m} = 1 \;.
 \end{align*}
Therefore there exists a constant $C > 1$ such that for any $x \in (0, 1]$ 
\begin{align*} 
    \frac{1}{C} a_m x^m \leq \overset{\infty}{\underset{n=1}{\sum}}a_nx^n \leq C a_m x^m
 \end{align*}
holds true. Hence, for \eqref{eq:onedim_analytic_a1} we obtain
\begin{align*} 
     \frac{\sigma^2}{C} \int_0^{1} \frac{1}{a_m x^m - p}\, \text{d}x 
     \leq \sigma^2 \mathlarger{\mathlarger{\int}}_0^{1} \frac{1}{\overset{\infty}{\underset{n=1}{\sum}}a_nx^n-p}\, \text{d}x 
     \leq C \sigma^2 \int_0^{1} \frac{1}{a_m x^m - p}\, \text{d}x \,. 
 \end{align*}
This result is equivalent to \eqref{eq:dim1_tool}, in the sense that it implies that the rate of divergence of $\langle V_{\infty}g, g \rangle$ is described in Table \ref{tab:onedim} with $m=\alpha$.
\end{proof}

\begin{remark} \label{rmk:general_g}
We have considered functions $g$ that are bounded. Through similar methods as the proof of Theorem \ref{thm:1-dim_tool}, a scaling law can be obtained for more general families of functions in $L^2(\mathcal{X})$. For example, suppose we consider $g=x^{-\gamma} \mathbbm{1}_{[0,\varepsilon]}$, $\gamma<\frac{1}{2}$, $f=f_\alpha$ and $\alpha>0$ then this yields
\begin{align*}
    \langle g, V_\infty g \rangle = \frac{\sigma^2}{2}\int_0^\varepsilon \frac{1}{x^{2\gamma}} \frac{1}{x^\alpha-p} \textnormal{d} x.
\end{align*}
Hence, for $0<2\gamma+\alpha<1$, we obtain
\begin{align*}
    \underset{p\to0^-}{\lim}\langle g, V_\infty g \rangle = \frac{\sigma^2}{2}\int_0^\varepsilon \frac{1}{x^{2\gamma+\alpha}} \textnormal{d} x<+\infty.
\end{align*}
Setting instead $2\gamma+\alpha\geq1$, we get
\begin{align*}
    \langle g, V_\infty g \rangle &= \frac{\sigma^2}{2}\int_0^\varepsilon \frac{1}{x^{2\gamma}} \frac{1}{x^\alpha-p} \textnormal{d} x\\
    &= \frac{\sigma^2}{2}(-p)^{-1-\frac{2\gamma}{\alpha}}\mathlarger{\int}_0^\varepsilon \frac{1}{\left(x (-p)^{-\frac{1}{\alpha}}\right)^{2\gamma}} \frac{1}{\left(x (-p)^{-\frac{1}{\alpha}}\right)^\alpha +1} \textnormal{d} x\\
    &= \frac{\sigma^2}{2}(-p)^{-1+\frac{1-2\gamma}{\alpha}}\int_0^{\varepsilon(-p)^{-\frac{1}{\alpha}}} \frac{1}{y^{2\gamma}} \frac{1}{y^\alpha +1} \textnormal{d} y \quad ,
\end{align*}
for $y=x(-p)^{-\frac{1}{\alpha}}$. The scaling law for $\langle g, V_\infty g \rangle$ in $p\to0^-$ is summarized in Table \ref{tab:extra}. This result generalizes the statement in \cite[Theorem 4.4]{KU} as for any analytic $f$ that satisfies \eqref{property_f} and $g$ in a dense subset of $L^2(\mathcal{X})$ an exact scaling  can be obtained.

\begin{table}[h!]
  \begin{center}
    \begin{tabular}{c|c}
      \specialrule{.1em}{.05em}{.05em}
      \textbf{Case} & \textbf{Scaling law for $p \to 0^{-}$} \\
      \hline
      $0< 2\gamma+\alpha < 1$ & $ 1 $ \\
      $2\gamma+\alpha = 1$ & $- \log(-p)$ \\
      $2\gamma+\alpha > 1$ & $(-p)^{-1 + \frac{1-2\gamma}{\alpha}}$ \\
      $2\gamma \to 1$\;\;\textnormal{or}\;\; $\alpha \to \infty$ & $(-p)^{-1}$\\
      \specialrule{.1em}{.05em}{.05em}
    \end{tabular}
    \caption{Scaling law in dimension $N=1$ for the function $f=f_{\alpha}(x)$ and $g=x^{-\gamma} \mathbbm{1}_{[0,\varepsilon]}$. }
    \label{tab:extra}
  \end{center}
\end{table}
\end{remark}

\begin{remark} \label{rmk:false_mirror}
We found that the scaling law of the time-asymptoptic variance, along an indicator function, of the solution of \eqref{sist} is shown in Table \ref{tab:onedim} for every function $f$ that satisfies \eqref{property_f} and \eqref{hyperb_conv} for $\alpha>0$.
 Clearly, this does not include all possible functions $f$. We take as example the case in which there exists $\delta\geq1$ such that $[-\delta,\delta]\subseteq\mathcal{X}$ and functions $f$ that converge at two different rates as $x$ approaches $0^-$ and $0^+$. We consider 
    \begin{align*} 
    f(x) =
    \begin{cases}
    f_1(x) & \text{for }\; 0 \geq x \in \mathcal{X} \, ,\\
    f_2(x) & \text{for }\; 0 < x \in \mathcal{X} \, ,\\
    \end{cases}
     \end{align*}
    for smooth functions $f_1:\mathcal{X}\cap\mathbb{R}_{\leq0}\rightarrow\mathbb{R}$ and $f_2:\mathcal{X}\cap\mathbb{R}_{\geq0}\rightarrow\mathbb{R}$ and $f$ that satisfies \eqref{property_f}. We assume $g(x)=\mathbbm{1}_{[ -1, 1 ]}(x)$ and by (\ref{eq:cov_integral}) we get
    \begin{align} \label{false_mirror}
        \langle V_{\infty}g, g \rangle = \int_{- 1}^{1} \frac{- \sigma^2}{2(f(x)+p)} \, \text{d}x 
        = \frac{\sigma^2}{2}\int_{-1}^{0} \frac{1}{-f_1(x)-p} \, \text{d}x + \frac{\sigma^2}{2}\int_{0}^{1} \frac{1}{-f_2(x)-p} \, \text{d}x \, .
    \end{align}
    We consider each summand separately and take the limit in $p$ to zero from below. If $\underset{x\rightarrow{0^+}}{\lim}f_2(x)<0$, then  
    \begin{align*}
        \lim_{p \to 0^{-}} \int_{0}^{1} \frac{1}{-f_2(x)-p} \, \text{d}x<+\infty
    \end{align*}
    and the scaling rate of $\langle V_{\infty}g, g \rangle$ is equivalent to the one given by by $\int_{-1}^{0} \frac{1}{-f_1(x)-p} \, \text{d}x$. Otherwise it is dictated by the highest rate given by the two summands. Such rates are shown in Table \ref{tab:onedim} if $f_1$ and $f_2$ are analytic or have the same order of convergence to $0$ as $f_\alpha$ for any $\alpha>0$.
\end{remark}

\section{Higher-dimensional cases} \label{sec:N-dim}

In the present section we obtain upper bounds for scaling of the covariance of the solutions of the system \eqref{sist} along chosen functions $g$. In this case we consider $N>1$, therefore assuming that the system \eqref{sist} is studied in higher spatial dimensions. We find that, for certain functions $f:\mathcal{X}\subseteq\mathbb{R}^N\to \mathbb{R}$, the early-warning signs display convergence of the variance along the mentioned directions in the square-integrable function space, $L^2(\mathcal{X})$.\\

For the remainder of this section, we assume that $f$ is analytic and satisfies \eqref{property_f} for $x_\ast=0\in\mathcal{X}\subseteq\mathbb{R}^N$. Hence, $f$ is of the form
\begin{align} \label{eq:highdim_analytic}
    f(x) = - \sum_{j \in \mathcal{C}} a_j x^j \quad \text{ for } x = (x_1, \dots, x_N)\in\mathcal{X}\,,
\end{align}
where $j$ is a multi-index, i.e., $x^j = \overset{N}{\underset{n=1}{\prod}} x_n^{i_n}$ with the collection $\mathcal{C}$ defined by the set
\begin{align*} 
     \mathcal{C} = \left\{j = (i_1, \ldots, i_N) \in \left\{\mathbb{N}\cup\{0\}\right\}^N \right\}\, .
 \end{align*} 
 The coefficients $a_j$ are real-valued and their signs satisfy \eqref{property_f} as discusses further in the paper. We assume that there exists $\varepsilon>0$ such that $[0, \varepsilon]^N\subseteq\mathcal{X}$, up to rescaling of the space variable $x$. Properties \eqref{property_f} imply that equation \eqref{eq:cov_integral} holds, hence for $g(x) = \mathbbm{1}_{[0, \varepsilon]^N}(x)$ we find
\begin{align*}
    \langle V_{\infty} g, g \rangle 
    = \sigma^2 \bigg\langle  \frac{-1}{2(f+p)} g, g \bigg\rangle 
    = \frac{\sigma^2}{2} \int_0^\varepsilon \cdots \int_0^\varepsilon \frac{-1}{f(x)+p} \, \text{d}x 
    = \frac{\sigma^2}{2} \mathlarger{\int}_0^\varepsilon \cdots \mathlarger{\int}_0^\varepsilon \frac{1}{\underset{j \in \mathcal{C}}{\sum}a_jx^j - p} \, \text{d}x \;.
\end{align*}

In the next proposition we prove that for any $f$ in dense subset of the analytic functions space, that satisfies \eqref{property_f}, the variance $\langle V_{\infty} g, g \rangle$ converges for $p\to 0^-$. 
\begin{prop} \label{conv_prop}
    Consider dimension $N>1$, and take the indices $\{i_k \}_{k \in \{ 1, \ldots, N\}}$ as a permutation of $\{ 1, \ldots, N\}$. Furthermore, suppose there exist two multi-indices $j_1,j_2\in\mathcal{C}$ such that $a_{j_1}, a_{j_2} > 0$ and that each of these multi-indices corresponds to the multiplication of only one $x_{i_1}$, resp. $x_{i_2}$, meaning that $j_1$, resp. $j_2$, are zero everywhere with the exception of the $i_1$-th, resp. $i_2$-th, position where they are $1$.
    Then it holds that 
    \begin{align*} 
       \lim_{p \to 0^{-}} \langle V_{\infty}g,g \rangle < \infty 
     \end{align*}
    for any $\varepsilon>0$ and $g(x) = \mathbbm{1}_{[0, \varepsilon]^N}(x)$ .
\end{prop}
\begin{proof}
    Without loss of generality, we assume that $j_1 = (1,0,0 \ldots)$ and $j_2 = (0,1,0 \ldots)$. Then we obtain 
\begin{align*}
    \langle V_{\infty}g,g \rangle &= \frac{\sigma^2}{2} \mathlarger{\int}_0^\varepsilon \cdots \mathlarger{\int}_0^\varepsilon \frac{1}{\underset{j \in \mathcal{C}}{\sum}a_jx^j - p}\, \text{d}x_N \cdots \, \text{d}x_1 \nonumber \\
    &= \frac{\sigma^2}{2} \mathlarger{\int}_0^\varepsilon \cdots \mathlarger{\int}_0^\varepsilon \frac{1}{a_{j_1} x_1 + a_{j_2} x_2 + \underset{j \in \mathcal{C} \setminus \{j_1, j_2\}}{\sum}a_jx^j - p}\, \text{d}x_N \cdots \, \text{d}x_1  \\
    &\leq C \frac{\sigma^2}{2} \int_0^\varepsilon \cdots \int_0^\varepsilon \frac{1}{x_1 + x_2 -p} \text{d}x_N \cdots \, \text{d}x_1 \, \nonumber
 \end{align*}
where the last inequality is satisfied by a constant $C>0$ and follows from the fact that $f$ converges to $0\in\mathbb{R}^2$ equivalently to $-a_1 x_1 - a_2 x_2$, the continuity of $f$, the compactness of the support of $g$ and $a_{j_1}, a_{j_2} > 0 > p$. Then we obtain  
\begin{align*} 
    & \lim_{p \to 0^{-}} \int_0^\varepsilon \cdots \int_0^\varepsilon \frac{1}{x_1 + x_2 - p} \text{d}x_N \cdots \, \text{d}x_1 
    = \varepsilon^{N-2} \lim_{p \to 0^{-}} \int_0^\varepsilon \int_0^\varepsilon \frac{1}{x_1 + x_2 - p} \text{d}x_2 \, \text{d}x_1\\
    &= \varepsilon^{N-2} \lim_{p \to 0^{-}} \int_0^\varepsilon \log (\varepsilon+x_1 - p) - \log (x_1 -p) \, \text{d}x_1 \\
    &= \varepsilon^{N-2} \lim_{p \to 0^{-}} \Big( (2\varepsilon-p)\log(2\varepsilon-p)-(2\varepsilon-p)
    -2(\varepsilon-p)\log(\varepsilon-p)+2(\varepsilon-p)
    -p\log(-p)+p \Big)\\
    &= \varepsilon^{N-2} \lim_{p \to 0^{-}} \Big( (2\varepsilon-p)\log(2\varepsilon-p)
    -2(\varepsilon-p)\log(\varepsilon-p)-p\log(-p) \Big)<+\infty
 \end{align*}
which concludes the proof.
\end{proof}

\begin{example} \label{ex:zero}
We study the trivial case in which we set $f$ in \eqref{sist} to be $f_\infty(x):= 0$ for any $x$ in an neighbourhood of $x_\ast$, in order to exclude it from the following computations. Whereas the function $f_\infty$ does not satisfy the assumptions on the sign of $f$ in \eqref{property_f}, it is an interesting and easy to study limit case, so we include it here. We obtain the variance along $g$ to be 
\begin{align*} 
    &\langle V_{\infty} g, g \rangle 
    = \frac{\sigma^2}{2} \int_0^\varepsilon\dots \int_0^\varepsilon \frac{1}{f_\infty(x_1, \dots, x_N)-p}\, \text{d}x_N \dots \, \text{d}x_1
    = \frac{\sigma^2}{2} \int_0^\varepsilon \dots \int_0^\varepsilon \frac{1}{-p}\, \text{d}x_N \dots \, \text{d}x_1 
    = -\frac{\sigma^2}{2p} \varepsilon^N \, .
 \end{align*}
Taking the limit in $p$ we obtain that 
\begin{align*}
    \lim_{p\to 0^{-}} \langle V_{\infty} g,g \rangle = \lim_{p\to 0^{-}}-\frac{\sigma^2}{2p} \varepsilon^N= \infty \,.
\end{align*}
Hence, we observe divergence with rate $\Theta\left(- p^{-1}\right)$ for $p$ approaching zero from below, if $f$ in \eqref{sist} is, locally, the null function.
\end{example}

\subsection{Upper bounds}
In the remainder of this section we find upper bounds for the scaling law of $\langle V_{\infty}g,g \rangle$ for the case in which the dimension $N$ of the domain of the function $f$ is bigger than $1$. From the construction of $f$ in \eqref{eq:highdim_analytic} we define the set of multi-indices
\begin{align}  \label{C+}
    \mathcal{C}^{+} := \Big\{ j=(i_1,\dots,i_N)\in \mathcal{C} \Big| \; a_d=0\;\;\forall d=(d_1,\dots,d_N)\in \mathcal{C} \;\;\text{s.t.}\;\; d_n\leq i_n\;\; \forall n\in\{1,\dots,N\} \;\;\text{and}\;\; d\neq j\Big\}\;.
 \end{align}

The following lemma introduces an upper bound of $\langle V_{\infty}g,g \rangle$ as $p\to0^-$. The scaling law induced by this upper bound is studied further below. 

\begin{lm} \label{lm:N-dim_up}
    Set $f(x)= - \underset{j \in \mathcal{C}}{\sum}a_j x^j$ for all $x \in \mathcal{X} \subseteq \mathbb{R}^N$, that satisfies \eqref{property_f}, $\{a_j\}_{j\in\mathcal{C}}\subset\mathbb{R}$, $\varepsilon>0$ and $Q=\operatorname{I}$. Fix $j_{\ast} \in \mathcal{C}^{+}$, defined in \eqref{C+}. Then the time-asymptotic covariance of the solution of \eqref{sist}, $V_\infty$, satisfies,
    \begin{align*}
        \langle V_{\infty}g, g \rangle 
        \leq
        \Theta\left( \int_0^\varepsilon \cdots \int_0^\varepsilon \frac{1}{a_{j_\ast} x^{j_\ast} -p} \, \text{d}x_1 \dots \text{d}x_N \right)\;,
    \end{align*}
    for $g(x)=\mathbbm{1}_{[0,\varepsilon]^N}$.
    
\end{lm}
\begin{proof}
Since we assume $f$ to be negative in $\mathcal{X}\setminus\{0\}$ and we consider the bounded domain $[0,\varepsilon]^N$, we know that $a_j>0$ for any $j\in\mathcal{C^+}$. In particular, we see that for a $1>C>0$, dependent on $\varepsilon$, and for any $x\in[0,\varepsilon]^N$
\begin{align*} 
    - f(x) = \sum_{j \in \mathcal{C}}a_j x^j 
    \geq C \sum_{j \in \mathcal{C}^{+}}a_j x^j 
    \geq C a_{j_{\ast}}x^{j_{\ast}}
 \end{align*}
for $j_{\ast} \in \mathcal{C}^{+}$. 
Hence, for the variance along $g$ we obtain
\begin{align} \label{eq:twodim_cov}
\langle V_{\infty}g, g \rangle 
= \frac{\sigma^2}{2} \mathlarger{\int}_{0}^\varepsilon \cdots\mathlarger{\int}_{0}^\varepsilon \frac{1}{\underset{j \in \mathcal{C}}{\sum} a_j x^{j} - p} \, \text{d}x_1 \dots \text{d}x_N
\leq \frac{\sigma^2}{2 C} \int_0^\varepsilon \cdots \int_0^\varepsilon \frac{1}{a_{j_\ast} x^{j_\ast} -p} \, \text{d}x_1 \dots \text{d}x_N \;.
 \end{align}
\end{proof}
Without loss of generality and for simplicity, we assume for the remainder of the section that $a_{j_{\ast}} = \frac{\sigma^2}{2} = 1$.

\begin{remark} \label{rmk:dim_reduction}
In case the multi-index $j_\ast$ has $N>k\in\mathbb{N}\cup\{0\}$ indices with value $0$, the analysis on the upper bound in \eqref{eq:twodim_cov} can be reduced to the case of $N-k$ spatial dimensions. In fact, assuming that $j_\ast=(i_1,\dots,i_{N-k},0,\dots,0)$, we obtain
\begin{align*}
    \int_0^\varepsilon \cdots \int_0^\varepsilon \frac{1}{x^{j_\ast}-p} \text{d}x_1\dots\text{d}x_N
    =\varepsilon^k \mathlarger{\mathlarger{\int}}_0^\varepsilon \cdots \mathlarger{\mathlarger{\int}}_0^\varepsilon \frac{1}{\overset{N-k}{\underset{n=1}{\prod}} x_n^{i_n}-p} \text{d}x_1\dots\text{d}x_{N-k}\;.
\end{align*}
We note that the case $N=k$ contradicts the assumption $f(x_\ast)=0$ and results in convergence of the upper bound as the bifurcation threshold is not reached for $p\to 0^-$.
\end{remark}
Due to Remark \ref{rmk:dim_reduction}, in the remaining subsections we consider $j_\ast$ with no elements equal to $0$.

\subsection{Two dimensions} \label{chap:dim2}

In this subsection, we find an upper bound for the scaling law of $\langle V_{\infty}g,g \rangle$ for $N=2$.  We therefore consider $j_\ast=(i_1,i_2)\in\mathbb{N}^2$. In the following theorem we outline how to obtain such upper bounds.

\begin{thm} \label{thm:2-dim_analytic}
    Set $f(x)= - \underset{j \in \mathcal{C}}{\sum}a_j x^j$ for all $x \in \mathcal{X} \subseteq \mathbb{R}^2$, that satisfies \eqref{property_f},  $\{a_j\}_{j\in\mathcal{C}}\subset\mathbb{R}$, $\varepsilon>0$ and $Q=\operatorname{I}$. Fix $j_\ast=(i_1,i_2)\in\mathcal{C}^+$ with $i_1,i_2>0$. Then there exists an upper bound to the time-asymptotic variance of the solution of \eqref{sist} along $g(x)=\mathbbm{1}_{[0,\varepsilon]^2}$, $\langle V_{\infty}g, g \rangle$, and its rate as $p\to 0^-$ is described in Table \ref{tab:twodim} in accordance to the value of $i_1$ and $i_2$.
\end{thm}

\begin{table}[h!]
  \begin{center}
    \begin{tabular}{c|c|c|c}
      \specialrule{.1em}{.05em}{.05em}
      \multirow{2}{*}{\textbf{Case}} & \multicolumn{3}{c}{\textbf{Scaling law for $p \to 0^{-}$}}\\
      \cline{2-4}
      & \textbf{$\mathfrak{A}$}& \textbf{$\mathfrak{B}$}& \textbf{Upper bound} \\
      \thickhline
      %$1= i_2 > i_1 = 0$ & $-$ & $-$ & $(-p)^{-1+ \frac{1}{i_2}}$ \\
      %$1 \neq i_2 > i_1 = 0$ & $-$ & $-$ & $- \log(-p) $\\
      $\mathbf{1: }\, i_2 > i_1 > 1$ & $(-p)^{-1 + \frac{1}{i_1}}$ & $(-p)^{-1 + \frac{1}{i_2}}$ & $(-p)^{-1 + \frac{1}{i_2}}$ \\
      $\mathbf{2: }\, i_2 > i_1 = 1$ &$-\log \left(-p\right) $ & $(-p)^{-1 + \frac{1}{i_2}}$ &$(-p)^{-1 + \frac{1}{i_2}}$ \\
      $\mathbf{3: }\, i_2 = i_1 > 1$ & $(-p)^{-1 + \frac{1}{i_1}}$ & $-(-p)^{-1 + \frac{1}{i_2}}\log (-p)$ & $-(-p)^{-1 + \frac{1}{i_2}}\log (-p)$ \\
      $\mathbf{4: }\, i_2 = i_1 = 1$ & $-\log \left(-p\right) $ & $\log^2(-p)$ & $\log^2(-p)$ \\
      \specialrule{.1em}{.05em}{.05em}
    \end{tabular}
    \caption{Upper bounds in dimension $N=2$ for different choices of indices $i_1, i_2$, ordered for simplicity. We indicate as $\mathfrak{A}$ and $\mathfrak{B}$ two summands whose sum gives the upper bound.}
    \label{tab:twodim}
  \end{center}
\end{table}

The proof of the theorem can be found in Appendix \ref{appendA}. Overall, we have obtained an upper bound for an analytic function $f$. In Table \ref{tab:twodim} we see that only the highest index $i_2$ affects directly the scaling law of the upper bound. However, its relation to the other index $i_1$ dictates the exact form. The limiting case for $i_1,i_2\to\infty$ has been discussed in Example \ref{ex:zero}. Lastly, we note that our bound may not be not sharp, as described in Proposition \ref{conv_prop} and as shown in the next example. 

\begin{example} \label{ex:log_div} We consider the function $f$ such that it satisfies \eqref{property_f}, $[0,1]^2\subseteq\mathcal{X}$ and there exists $C>0$ for which
\begin{align*}
    C^{-1} \left( - x_1^2-x_2^2 \right) \leq f(x) \leq C \left( - x_1^2-x_2^2 \right) \, ,
\end{align*}
for $0\leq x_1,x_2 \leq 1$. This means that for the covariance operator we obtain
\begin{align*} 
    \langle V_{\infty}g, g \rangle &\leq C^{-1} \int_{0}^1\int_0^1 \frac{1}{x_1^2+x_2^2-p} \, \text{d}x_1 \, \text{d}x_2   \leq C^{-1} \int \int_{D} \frac{1}{x_1^2 + x_2^2 -p} \, \text{d}x_1 \, \text{d}x_2 \quad,
 \end{align*}
where $D$ denotes the circle of radius $\sqrt{2} $ centered at the origin. We study then the integral in polar coordinates and obtain
\begin{align*} 
    \int \int_{D} \frac{1}{x_1^2 + x_2^2 -p} \, \text{d}x_1 \, \text{d}x_2  = \int_0^{2 \pi} \int_{0}^{\sqrt{2} } \frac{r}{r^2 - p} \, \text{d}r \, \text{d} \theta \, . 
 \end{align*}
Substituting $r^{\prime} = r^2 - p$ we get
\begin{align*} 
    &\int_0^{2 \pi} \int_{0}^{\sqrt{2} } \frac{r}{r^2 - p} \, \text{d}r \, \text{d} \theta 
    = \frac{1}{2} \int_0^{2 \pi} \int_{-p}^{2 -p} \frac{1}{r^{\prime}} \, \text{d}r^{\prime} \, \text{d} \theta
    = \pi \Big( \log(2 -p) - \log(-p) \Big)
    =\Theta \Big(- \log \left(-p \right) \Big)\, . 
 \end{align*}
A lower bound can be easily obtained through 
\begin{align*} 
    \langle V_{\infty}g, g \rangle &\geq C \int_{0}^1\int_0^1 \frac{1}{x_1^2+x_2^2-p} \, \text{d}x_1 \, \text{d}x_2   
    \geq C \int \int_{\tilde{D}} \frac{1}{x_1^2 + x_2^2 -p} \, \text{d}x_1 \, \text{d}x_2
    =\Theta \Big(- \log \left(-p \right) \Big) \quad,
 \end{align*}
 for $\tilde{D}=\left\{(x,y) \;\big|\; x^2+y^2\leq 1 \;,\; x,y>0\right\}$.
Hence we find divergence of $\langle g,V_\infty g\rangle$, for $p$ approaching zero from below, of rate $\Theta \Big(- \log \left(-p \right) \Big)$, whereas the upper bounds in Theorem \ref{thm:2-dim_analytic} assume scaling law $\Theta\left((-p)^{-\frac{1}{2}} \right)$.
\end{example}

\subsection{Three dimensions} \label{chap:dim3}

We now assume $N=3$ and therefore $f:\mathcal{X}\subseteq\mathbb{R}^3\to\mathbb{R}$ and $g=\mathbbm{1}_{[0,\varepsilon]^3}$. The following theorem provides an upper bound for the scaling law of the corresponding early-warning sign $\langle g, V_\infty g \rangle$.

\begin{thm} \label{thm:3-dim_analytic}
    Set $f(x)= - \underset{j \in \mathcal{C}}{\sum} a_j x^j$ for all $x \in \mathcal{X} \subseteq \mathbb{R}^3$, that satisfies \eqref{property_f},  $\{a_j\}_{j\in\mathcal{C}}\subset\mathbb{R}$, $\varepsilon>0$ and $Q=\operatorname{I}$. Fix $j_\ast=(i_1,i_2,i_3)\in\mathcal{C}^+$ with $i_1,i_2, i_3>0$. Then there exists an upper bound to the time-asymptotic variance of the solution of \eqref{sist} along $g(x)=\mathbbm{1}_{[0,\varepsilon]^3}$, $\langle V_{\infty}g, g \rangle$, and it has a scaling law bound as $p\to 0^-$ described in Table \ref{tab:threedim} depending upon the values of $i_1$, $i_2$ and $i_3$.
\end{thm}

\begin{table}[h!]
  \begin{center}
    %\resizebox{\textwidth}{!}{
    %\adjustbox{max width=\textwidth}{
    \begin{tabular}{c|c|c|c}
      \specialrule{.1em}{.05em}{.05em}
      %\multicolumn{2}{c|}{\multirow{2}{*}{\textbf{Case}}} & \multicolumn{3}{c}{\textbf{Asymptotic behaviour for $p \to 0^{-}$}}\\
      \multirow{2}{*}{\textbf{Case}} & \multicolumn{3}{c}{\textbf{Scaling law for $p \to 0^{-}$}}\\
      \cline{2-4}
      %\multicolumn{2}{c|}{} 
      &\textbf{$\mathfrak{C}$}& \textbf{$\mathfrak{D}$}& \textbf{Upper bound} \\
      \thickhline
      $\mathbf{1:}\, i_3 > i_2 > i_1 > 1$ & $(-p)^{-1+ \frac{1}{i_2}} $ & $(-p)^{-1 + \frac{1}{i_3}} $ & $(-p)^{-1 + \frac{1}{i_3}}$ \\ 
      $\mathbf{2:}\, i_3 > i_2 = i_1 > 1$ & $-(-p)^{-1+\frac{1}{i_2}} \log(-p)$ & $(-p)^{-1+ \frac{1}{i_3}}$ &  $(-p)^{-1+ \frac{1}{i_3}}$ \\ 
      $\mathbf{3:}\, i_3 = i_2 > i_1 > 1$ & $(-p)^{-1+ \frac{1}{i_2}} $ & $-(-p)^{-1 + \frac{1}{i_3}} \log(-p)$ & $-(-p)^{-1 + \frac{1}{i_3}} \log(-p)$ \\
      $\mathbf{4:}\, i_3 = i_2 = i_1 > 1$ & $-(-p)^{-1+\frac{1}{i_2}} \log(-p)$ & $(-p)^{-1+ \frac{1}{i_3}}\log^2(-p)$ & $(-p)^{-1+ \frac{1}{i_3}}\log^2(-p)$\\
      $\mathbf{5:}\, i_3 > i_2 > i_1 = 1$ & $(-p)^{-1+ \frac{1}{i_2}}$ & $(-p)^{-1+\frac{1}{i_3}}$ & $(-p)^{-1+ \frac{1}{i_3}} $\\ 
      $\mathbf{6:}\, i_3 > i_2 = i_1 = 1$ & $\log^2 (-p)$ & $(-p)^{-1+ \frac{1}{i_3}}$ & $(-p)^{-1+ \frac{1}{i_3}}$ \\
      $\mathbf{7:}\, i_3 = i_2 > i_1 = 1$ & $(-p)^{-1+ \frac{1}{i_2}}$ & $-(-p)^{-1+ \frac{1}{i_3}} \log(-p)$ & $-(-p)^{-1+ \frac{1}{i_3}} \log(-p)$ \\
      $\mathbf{8:}\, i_3 = i_2 = i_1 = 1$ & $\log^2 (-p)$ & $- \log^3(-p)$ & $- \log^3(-p)$ \\
      \specialrule{.1em}{.05em}{.05em}
    \end{tabular}
    %}
    \caption{Upper bounds in dimension $n=3$ for different choices of indices $i_1, i_2$ and $i_3$, ordered for simplicity.  We denote as $\mathfrak{C}$ and $\mathfrak{D}$ two values whose sum is the upper bound.}
    \label{tab:threedim}
  \end{center}
\end{table}

The theorem is proven in Appendix \ref{appendB}. We observe that if $i_3$ is strictly greater than the other indices, $i_1$ and $i_2$, we find divergence with the upper bound $\Theta\left((-p)^{-1 + \frac{1}{i_3}}\right)$. If the highest index value is given by two of the indices, i.e., $i_3=i_2>i_1$, then the upper bound we obtain for the scaling is $\Theta\left(-(-p)^{-1 + \frac{1}{i_3}}\log(-p)\right)$. We see also that in the case all three indices are equal but greater than $1$, i.e., $i_3=i_2=i_1>1$, then the upper bound is $\Theta\left((-p)^{-1 + \frac{1}{i_3}}\log^2(-p)\right)$. Lastly, setting all the indices equal to $1$, then the described rate is $\Theta\left(-\log^3(-p)\right)$. In Example \ref{ex:zero} the limit case $i_1,i_2,i_3\to \infty$ is covered.

\section{Generalizations and applications} \label{sec:gen}

We now generalize and discuss the main results obtained above with a focus on the impact of relaxing the hypotheses in Section \ref{sec:preliminaries}. In particular, one may obtain Theorem \ref{thm:1-dim_analytic}, Theorem \ref{thm:2-dim_analytic} and Theorem \ref{thm:3-dim_analytic} under slightly more general choices of functions $f, g$ and operator $Q$. We also study the theorems for $\operatorname{T}_f$ with complex spectrum. Lastly, we discuss the case in which several concrete linear operators, $A$, with continuous spectrum are encountered in the drift term of \eqref{sist}. 

%The different generalizations can be assumed in conjunction.

\begin{itemize}[label=$\blacktriangleright$]
\item
The condition of uniqueness of $x_\ast$ such that $f(x_\ast)=0$ is required only in a local sense. In fact, we could assume the existence of $\mathcal{Z}\subset\mathcal{X}$ for which any $x\in\mathcal{Z}$ satisfies $f(x)=0$ and $\text{dist} (x_\ast, x) > 0$. Such choice would trivially leave unchanged the scaling laws described in Theorem \ref{thm:1-dim_analytic}, Theorem \ref{thm:2-dim_analytic} and Theorem \ref{thm:3-dim_analytic}, for functions $g$ such that the support of $g$ has empty intersection with $\mathcal{Z}$.\\
The analytic behaviour of the non-positive function $f$ presented in Theorem \ref{thm:1-dim_analytic}, Theorem \ref{thm:2-dim_analytic} and Theorem \ref{thm:3-dim_analytic} can be assumed only in a neighbourhood of the roots of the function, while $f$ satsfies the integrability property in \eqref{property_f} and the sign properties previously described.

\item
In Theorem \ref{thm:1-dim_analytic}, Theorem \ref{thm:2-dim_analytic} and Theorem \ref{thm:3-dim_analytic}, the support of the function $g$ is assumed to be $[0,\varepsilon]^N$ and $g$ to be an indicator function, yet the scaling laws in Table \ref{tab:onedim}, Table \ref{tab:twodim} and Table \ref{tab:threedim} are applicable for a more general choice.\\ 
First, similarly to Remark \ref{rmk:false_mirror}, the upper bounds described in Theorem \ref{thm:2-dim_analytic} and Theorem \ref{thm:3-dim_analytic} can be shown to hold with $g=\mathbbm{1}_{[-\varepsilon,\varepsilon]^N}$, in the sense that the hypercube $[-\varepsilon,\varepsilon]^N\subseteq\mathcal{X}$ that describes the support of $g$ can be split in $2^N$ hypercubes on which such statements have been proven, up to reparametrization and rescaling of the spatial variable $x$, and of which only the highest rate of divergence dictates the order assumed by the upper bound of $\langle g, V_\infty g\rangle$.\\
The shape of the support can also be generalized. Set $g=\mathbbm{1}_{\mathcal{S}}$ for $[-\varepsilon,\varepsilon]^N\subseteq\mathcal{S}\subseteq\mathcal{X}$, then the integral
\begin{align*}
    \langle V_{\infty} g, g \rangle
    = \frac{\sigma^2}{2} \int_{\mathcal{S}} \frac{-1}{f(x)+p} \, \text{d}x 
    = \frac{\sigma^2}{2} \int_{\mathcal{S}\setminus[-\varepsilon,\varepsilon]^N} \frac{-1}{f(x)+p} \, \text{d}x 
    + \frac{\sigma^2}{2} \int_{[-\varepsilon,\varepsilon]^N} \frac{-1}{f(x)+p} \, \text{d}x 
\end{align*}
assumes equivalent rate of the divergence to the second summand in the right-hand side of the equation. That is implied by the fact that the first summand converges as $p\to 0^-$ by construction.\\
The scaling law is also unchanged with $g\in L^2(\mathcal{X})$ such that $g$ and $g^{-1}$ are bounded from above in a neighbourhood of $x_\ast$ as, for any $p<0$,
\begin{align*}
    C^{-1} \frac{\sigma^2}{2} \int_{\mathcal{S}} \frac{-1}{f(x)+p} \, \text{d}x
    \leq \langle V_{\infty} g, g \rangle
    \leq C \frac{\sigma^2}{2} \int_{\mathcal{S}} \frac{-1}{f(x)+p} \, \text{d}x \; ,
\end{align*}
for a constant $C>1$ depending on $g$ and on the choice of $\mathcal{S}$.\\
Lastly, assume $g_1,g_2\in L^2(\mathcal{X})$ which are continuous in a neighbourhood $\mathcal{S}$ of $x_\ast$ and such that $g_1,g_2,g_1^{-1}$ and $g_2^{-1}$ are bounded from above in $\mathcal{S}$. We note that such set of functions is dense in $L^2(\mathcal{X})$. For simplicity set $g_1(x_\ast) g_2(x_\ast)>0$. The scaling law of the scalar product in \eqref{eq:cov_integral} can be obtained through
\begin{align*}
    C^{-1} \frac{\sigma^2}{2} \int_{\mathcal{S}} \frac{-1}{f(x)+p} \, \text{d}x
    \leq \langle V_{\infty} g_1, g_2 \rangle
    \leq C \frac{\sigma^2}{2} \int_{\mathcal{S}} \frac{-1}{f(x)+p} \, \text{d}x \; ,
\end{align*}
for $C>1$ and any $p$ close to $0$, thus enabling a study of the covariance operator $V_\infty$. 

\item
A third generalization of Theorem \ref{thm:1-dim_analytic}, Theorem \ref{thm:2-dim_analytic} and Theorem \ref{thm:3-dim_analytic} is given by relaxing the assumptions for $Q$. We can assume $Q$ and its inverse on $L^2(\mathcal{X})$, $Q^{-1}$, to have bounded eigenvalues from above. The covariance operator (\cite[Theorem 5.2]{DP}) takes the form
\begin{align*}
    V_{\infty} = \sigma^{2} \int_{0}^{\infty} e^{\tau (\operatorname{T}_f+p)} Q e^{\tau (\operatorname{T}_f+p)}\, \text{d}\tau \;.
\end{align*}
Hence we get 
\begin{align*} 
    \langle V_{\infty} g, g \rangle 
    &= \sigma^2 \int_{0}^{\infty} \langle \sqrt{Q}e^{t(\operatorname{T}_f+p)}g, \sqrt{Q}e^{t (\operatorname{T}_f+p)}g \rangle \, \text{d}t 
    = \sigma^2 \int_{0}^{\infty}  \lVert \sqrt{Q}e^{t(\operatorname{T}_f+p)}g \rVert^{2} \, \text{d}t \;, 
 \end{align*}
for any $g\in L^2(\mathcal{X})$. Since we have assumed that $Q$ and its inverse have bounded eigenvalues from above, the scalar product is also bounded from below and above as we have
\begin{align} \label{Q_gen}
    \inf_{n}{q_n \sigma^2 \int_0^{\infty} \lVert e^{(\operatorname{T}_f+p)t}g \rVert^{2} \, \text{d}t} \leq \sigma^2 \int_{0}^{\infty}  \lVert \sqrt{Q}e^{(\operatorname{T}_f+p)t}g \rVert^{2} \, \text{d}t \leq \sup_n{q_n \sigma^2 \int_0^{\infty} \lVert e^{(\operatorname{T}_f+p)t}g \rVert^{2} \, \text{d}t} \, ,
 \end{align}
 with $\{q_n\}_{n\in\mathbb{N}}$ the eigenvalues of $Q$. This implies that we have found an upper and lower bound for $\langle V_{\infty}g, g \rangle$ whose scaling law is controlled in accordance with the results in Theorem \ref{thm:1-dim_analytic}, Theorem \ref{thm:2-dim_analytic} and Theorem \ref{thm:3-dim_analytic}. \\
Assuming instead $Q$ to be bounded, the validity of the rates of the upper bound in Theorem \ref{thm:2-dim_analytic} and Theorem \ref{thm:3-dim_analytic} is maintained as the second inequality in \eqref{Q_gen} holds.

\item
We now consider the case in which the spectrum of $A = \operatorname{T}_f$ has complex values, i.e., we choose a function $f : \mathcal{X} \subseteq \mathbb{R}^{N} \to \mathbb{C}$ and the solution of \eqref{sist} as $u\in L^2(\mathcal{X};\mathbb{C)}$. We suppose also that $\text{Re}(f)$ satisfies \eqref{property_f}. The scalar product of the covariance operator takes then the form 
\begin{align*} 
    \langle V_{\infty}g, g \rangle 
    &= \int_{\mathcal{X}} \overline{g(x)} \int_{0}^{\infty} e^{(\overline{f(x)} + p) t} e^{(f(x) + p) t} g(x) ~\textnormal{d}t ~\textnormal{d}x 
    = -\int_{\mathcal{X}} \lvert g(x) \rvert^2 \frac{\sigma^2}{2(\text{Re}(f(x))+p)}\, \text{d}x \, .
 \end{align*}
for any $g\in L^2(\mathcal{X};\mathbb{C})$. Assuming that $\text{Re}(f(x))$ is analytic, we conclude that the results for the real spectrum achieved for the one-, two- and three-dimensional case hold true applying such theorems to the function $\text{Re}(f(x))$.

\item
Lastly we study the system
\begin{align} \label{new_sist}
    \begin{cases}
    \text{d}u(x,t) = \left(A+p\right)\; u(x,t) \, \text{d}t + \sigma \text{d}W(t)\\
    u(\cdot,0) = u_0
    \end{cases}
\end{align}
for $A$ a linear self-adjoint operator in $L^2(\mathcal{X})$ with non-positive spectrum. The spectral theorem \cite[Theorem 10.10]{hall2013quantum} implies the existence of a $\sigma$-finite measure space $(\mathcal{X}_0,\mu)$, a measurable function $\tilde{f}:\mathcal{X}_0\to \mathbb{R}$ and a unitary map $U:L^2(\mathcal{X})\to L^2(\mathcal{X}_0,\mu)$ such that
\begin{equation*}
    U(L^2(\mathcal{X}))=\left\{ \tilde{g}\in L^2(\mathcal{X}_0,\mu) \; | \; \tilde{f}\tilde{g}\in L^2(\mathcal{X}_0,\mu) \right\}
\end{equation*}
and 
\begin{equation*}
    U A U^{-1}=\operatorname{T}_{\tilde{f}} \;,
\end{equation*}
for $\operatorname{T}_{\tilde{f}}: L^2(\mathcal{X}_0,\mu) \to L^2(\mathcal{X}_0,\mu)$ a multiplication operator for $\tilde{f}$. Assuming $e^{A+p}$ to have finite trace, the time-asymptotic covariance operator for the solution of \eqref{new_sist} is
\begin{align*} 
    \langle V_{\infty} g, g \rangle &= \sigma^2 \int_{0}^{\infty} \langle  e^{2 t (A+p)} g , g\rangle \, \text{d}t \;,
 \end{align*}
 for $Q=\operatorname{I}$ and $g\in L^2(\mathcal{X})$. Labeling $\tilde{g}=U g$ we obtain through Fubini's Theorem that
\begin{align*} 
    \langle V_{\infty} g, g \rangle 
    &= \sigma^2 \int_{0}^{\infty} \langle U^{-1} e^{2 t (\operatorname{T}_{\tilde{f}}+p)} U g , g\rangle \, \text{d}t 
    = \sigma^2 \int_{0}^{\infty} \langle e^{2 t (\operatorname{T}_{\tilde{f}}+p)} \tilde{g} , \tilde{g} \rangle_{L^2(\mathcal{X}_0,\mu)} \, \text{d}t \\
    &= \int_{\mathcal{X}_0} \tilde{g}(x) \int_{0}^{\infty} e^{2 t (\tilde{f}(x) + p)} \tilde{g}(x)~ \textnormal{d}t~ \textnormal{d}\mu(x) 
    = - \int_{\mathcal{X}_0} \tilde{g}(x)^2 \frac{\sigma^2}{2(\tilde{f}(x)+p)}\, \text{d}\mu(x) \, .
 \end{align*}
 The operators $A$ and $\operatorname{T}_{\tilde{f}}$ have the same spectrum. We can assume that there exists an interval $(-\delta,0]$ in the continuous spectrum of $A$, which yields that $\mu(\{\tilde{f}=0\})=0$ and that there exists at least a point $x_\ast\in\mathcal{X}_0$ such that $\tilde{f}(x_\ast)=0$. From the Stone-Weierstrass theorem we know that for $g$ in a dense subset of $L^2(\mathcal{X})$ the associated $\tilde{g}\in L^2(\mathcal{X}_0,\mu)$ assumes bounded and non-null values $\mu$-a.e. in a neighbourhood of each root $x_\ast$. The scaling law for the case $N=1$ can be studied similarly to Remark \ref{rmk:general_g} given insights about the measure $\mu$ and assuming $\tilde{f}$ analytic (\cite{hall2013quantum}). Such assumption is not restrictive for $A$ bounded since such condition implies that $\tilde{f}$ is bounded \cite[Theorem 7.20]{hall2013quantum}. In fact, fix $x_\ast$, a neighbourhood $\mathcal{X}_1\subseteq\mathcal{X}_0$ and define
 \begin{align*}
     S_p&=\{\tilde{f}:\mathcal{X}_1\to\mathbb{R} \;|\; \tilde{f}\;\text{is polynomial}\;,\;\tilde{f}(x_\ast)=0\}\;,\\
     S_a&=\{\tilde{f}:\mathcal{X}_1\to\mathbb{R} \;|\; \tilde{f}\;\text{is analytic}\;,\;\tilde{f}(x_\ast)=0\}\;,\\
     S_c&=\{\tilde{f}:\mathcal{X}_1\to\mathbb{R} \;|\; \tilde{f}\;\text{is continuous}\;,\;\tilde{f}(x_\ast)=0\}\;,\\
     S_b&=\{\tilde{f}:\mathcal{X}_1\to\mathbb{R} \;|\; \tilde{f}\;\text{is bounded}\;,\;\tilde{f}(x_\ast)=0\}\;,\\
     S_l&=\{\tilde{f}:\mathcal{X}_1\to\mathbb{R} \;|\; \tilde{f}\in L^2(\mathcal{X}_1)\;,\;\tilde{f}(x_\ast)=0\}\;.
 \end{align*}
  Stone-Weierstrass Theorem states that $S_p$, and thus $S_a$, is dense in $S_c$. Also, $S_c$ is dense in $S_l$, and therefore in $S_b$. Hence elements in $S_a\cap\{\tilde{f}:\mathcal{X}_1\to\mathbb{R} \;|\; \tilde{f}(x)<0 \;\text{for}\;x\neq x_\ast\}$ can approximate functions in $S_b\cap\{\tilde{f}:\mathcal{X}_1\to\mathbb{R} \;|\; \tilde{f}(x)<0 \;\text{for}\;x\neq x_\ast\}$. Such sets of functions can be treated due to \cite[Proposition 7.21]{hall2013quantum}.
  \end{itemize}

  \begin{example}
 Fix a non-positive function $f:\mathbb{R}\to \mathbb{R}$ that satisfies \eqref{property_f}. Consider the linear operator $A:\mathcal{D}(A)\to L^2(\mathbb{R})$ such that for any $g\in \mathcal{D}(A)$ it holds
  \begin{align*}
    A g(x)=f\ast g(x) \quad,\quad \mathcal{D}(A):=\left\{ g\;\;\text{Lebesgue measurable}\;\Big|\; f \ast g\in L^2(\mathcal{X}) \right\}\quad,
  \end{align*}
  where $\ast$ denotes convolution. We want to study the variance of the solution of the system \eqref{new_sist} for $Q$ bounded with bounded inverse and $u_0\in\mathcal{D}(A)$.  We define the Fourier transform $\mathcal{F}:L^2(\mathbb{R})\to L^2(\mathbb{R})$, which is a unitary map. Assume $g\in L^2(\mathbb{R})$, then
  \begin{align*}
      \langle g, V_\infty g \rangle 
      &= \sigma^2 \int_{0}^{\infty} \langle  e^{t (A+p)}Qe^{t (A+p)}g , g\rangle \, \text{d}t
      = \sigma^2 \int_{0}^{\infty}  \lVert \sqrt{Q}e^{t(A+p)}g \rVert^{2} \, \text{d}t
      = \Theta \left( \int_{0}^{\infty}  \lVert e^{t(A+p)}g \rVert^{2} \, \text{d}t \right)\\
      &= \Theta \left( \int_{0}^{\infty} \lVert \mathcal{F}^{-1} e^{t(\operatorname{T}_f+p)} \mathcal{F} g \rVert^{2} \, \text{d}t \right)
      = \Theta \left( \int_{0}^{\infty} \lVert e^{t(\operatorname{T}_f+p)} \mathcal{F} g \rVert^{2} \, \text{d}t \right)\;.
  \end{align*}
  The scaling law of the variance can thus be computed, or compared, through Theorem \ref{thm:1-dim_tool} for $f$ that satisfies \eqref{hyperb_conv} for $\alpha>0$ and for $g$ in a dense subset of $L^2(\mathbb{R})$.
  \end{example}
  
  \begin{example}
 Consider the self-adjoint operator $A:H^{2m}(\mathbb{R}) \to L^2(\mathbb{R})$ for $m\in\mathbb{N}$, such that
  \begin{align*}
      A g(x) =(-1)^{(m-1)} \partial_x^{2m} g(x) \;,
  \end{align*}
  where $\partial_x$ denotes the weak derivative on $H^1(\mathbb{R})$ and we assume $g\in H^{2m}(\mathbb{R})$. We want to study the variance of the solution of \eqref{new_sist} for $u_0\in H^{2m}(\mathbb{R})$, $p<0$ and $Q$ a bounded operator with bounded inverse. For $g\in H^{2m}(\mathbb{R})$ and $f(k)=-k^{2m}$ for any $k\in\mathbb{R}$, we find that
  \begin{align*}
      \langle g, V_\infty g \rangle 
      &= \sigma^2 \int_{0}^{\infty} \langle  e^{t (A+p)}Qe^{t (A+p)}g , g\rangle \, \text{d}t
      = \Theta \left( \int_{0}^{\infty}  \lVert e^{t(A+p)}g \rVert^{2} \, \text{d}t \right)\\
      &= \Theta \left( \int_{0}^{\infty} \lVert \mathcal{F}^{-1} e^{t(\operatorname{T}_{f}+p)} \mathcal{F} g \rVert^{2} \, \text{d}t \right)
      = \Theta \left( \int_{0}^{\infty} \lVert e^{t(\operatorname{T}_f+p)} \mathcal{F} g \rVert^{2} \, \text{d}t \right).
  \end{align*}
  From Theorem \ref{thm:1-dim_tool}, the scaling law as $p\to 0^-$ for $g$ in a dense subset of $H^{2m}(\mathbb{R})$ is therefore $\Theta\left(p^{-1+\frac{1}{2m}} \right)$.
  \end{example}
  
  \begin{example}
  The previous example can easily be generalized for different self-adjoint linear differential operators through Theorem \ref{thm:1-dim_analytic}. Consider for instance the Swift-Hohenberg equation~\cite{KU2} on the real line, linearized around the trivial solution $u(x)=0$ for any $x\in\mathbb{R}$ and pertubed by additive white noise. It thus takes the form of \eqref{new_sist} with $u_0\in H^{4}(\mathbb{R})$, $Q=\operatorname{I}$ and $A:H^{4}(\mathbb{R}) \to L^2(\mathbb{R})$ such that 
  \begin{align*}
      A g(x) =-\left( 1+\partial_x^{2} \right)^2 g(x) \;,
  \end{align*}
  for any $g\in H^{4}(\mathbb{R})$. The scaling law of the variance of such a solution along  $g$ in a dense subset of $H^4(\mathbb{R})$ is given therefore by
    \begin{align*}
      \langle g, V_\infty g \rangle 
      &= \sigma^2 \int_{0}^{\infty} \langle  e^{t (A+p)} e^{t (A+p)}g , g\rangle \, \text{d}t
      = \Theta \left( \int_{0}^{\infty} \lVert e^{t(\operatorname{T}_f+p)} \mathcal{F} g \rVert^{2} \, \text{d}t \right)\;,
  \end{align*}
  with $f(k)=-\left( 1-k^2 \right)^2$ for any $k\in\mathbb{R}$. By Taylor expansion of $f$ at $k=\pm1$, one concludes that the rate of divergence is given by $\Theta\left((-p)^{-\frac{1}{2}} \right)$.
  \end{example}
  
  %\begin{example}
  %Consider $A:H^{4}(\mathbb{R}^2) \to L^2(\mathbb{R}^2)$ on the Sobolev space $H^{4}(\mathbb{R}^2)$. We consider the operator $A$ as
  %\begin{align*}
  %    A g (x) = \left(\partial_{x_1}^{2}+\partial_{x_2}^2-\partial_{x_1}^{4}-\partial_{x_2}^{4}\right) g(x) \;,
  %\end{align*}
  %for any $g\in H^{4}(\mathbb{R}^2)$ and $x=(x_1,x_2)$. We study the variance of the solution of \eqref{new_sist} for $u_0\in %H^{4}(\mathbb{R}^2)$, $p<0$ and $Q$ a bounded operator with bounded inverse. We consider $g\in H^{4}(\mathbb{R})$ and $f_i(k)=k_i^{2}$ %for $i\in\{1,2\}$ and $k=(k_1,k_2)\in\mathbb{R}^2$. Let  $\mathcal{F}$ denote the Fourier transform on $L^2(\mathbb{R}^2)$. Then one %calculates
  %\begin{align*}
  %    \langle g, V_\infty g \rangle 
  %    &= \sigma^2 \int_{0}^{\infty} \langle  e^{t (A+p)}Qe^{t (A+p)}g , g\rangle \, \text{d}t
  %    = \Theta \left( \int_{0}^{\infty}  \lVert e^{t(A+p)}g \rVert^{2} \, \text{d}t \right)\\
  %    &= \Theta \left( \int_{0}^{\infty} \lVert \mathcal{F}^{-1} e^{t(\operatorname{T}_{-f_1-f_2-f_1^2-f_2^2}+p)} \mathcal{F} g \rVert^{2} \, \text{d}t \right)
  %    = \Theta \left( \int_{0}^{\infty} \lVert e^{t(\operatorname{T}_{-f_1-f_2-f_1^2-f_2^2}+p)} \mathcal{F} g \rVert^{2} \, \text{d}t %\right)\;.
  %\end{align*}
  %From Example \ref{ex:log_div}, the rate of divergence as $p\to 0^-$ for almost any $g\in H^{4}(\mathbb{R}^2)$ is $\Theta\left(-\log(-p) \right)$.
  %\end{example}
  
   \begin{example}
  Consider $A:H^{4}(\mathbb{R}^2) \to L^2(\mathbb{R}^2)$ as the Swift-Hohenberg operator on two spatial dimensions, i.e.,
  \begin{align*}
      A g (x) = -\left( 1+\partial_{x_1}^{2}+\partial_{x_2}^{2} \right)^2 g(x) \;,
  \end{align*}
  for any $g\in H^{4}(\mathbb{R}^2)$ and $x=(x_1,x_2)$. We study the variance of the solution of \eqref{new_sist} for $u_0\in H^{4}(\mathbb{R}^2)$, $p<0$ and $Q$ a bounded operator with bounded inverse. We consider $g\in H^{4}(\mathbb{R})$ and $f_i(k)=k_i$ for $i\in\{1,2\}$ and $k=(k_1,k_2)\in\mathbb{R}^2$. Let  $\mathcal{F}$ denote the Fourier transform on $L^2(\mathbb{R}^2)$ and $\hat{g}=\mathcal{F} g$. Then one calculates
  \begin{align*}
      \langle g, V_\infty g \rangle 
      &= \sigma^2 \int_{0}^{\infty} \langle  e^{t (A+p)}Qe^{t (A+p)}g , g\rangle \, \text{d}t
      = \Theta \left( \int_{0}^{\infty}  \lVert e^{t(A+p)}g \rVert^{2} \, \text{d}t \right)\\
      &= \Theta \left( \mathlarger{\mathlarger{\int}}_{0}^{\infty} \left\lVert \mathcal{F}^{-1} e^{t\left(\operatorname{T}_{-\left(1-f_1^2-f_2^2\right)^2}+p\right)} \mathcal{F} g \right\rVert^{2} \, \text{d}t \right)
      = \Theta \left( \mathlarger{\mathlarger{\int}}_{0}^{\infty} \left\lVert e^{t\left(\operatorname{T}_{-\left(1-f_1^2-f_2^2\right)^2}+p\right)} \hat{g} \right\rVert^{2} \, \text{d}t \right)\;.
  \end{align*}
  We know that the scaling law as $p\to 0^-$ of such scalar product for a dense set of functions $g$ in $H^4(\mathbb{R}^2)$ is equivalent to the one assuming $g$ such that $\hat{g}=\mathbbm{1}_D$ where $D$ denotes the circle of radius $\sqrt{2}$ centered at the origin in $\mathbb{R}^2$. Through Fubini's Theorem we can study the integral on polar spatial coordinates as
  \begin{align*}
      &\int_{0}^{\infty} \int \int_D e^{t\left(-\left(1-k_1^2-k_2^2\right)^2+p\right)} \text{d}k \, \text{d}t
      =\int \int_D \frac{1}{\left(1-k_1^2-k_2^2\right)^2-p} \text{d}k
      =\int_0^{2\pi} \int_0^{\sqrt{2}} \frac{r}{\left(1-r^2\right)^2-p} \text{d}r \, \text{d}\theta \;.
  \end{align*}
  Introducing $r'=1-r^2$ and $r''=r'(-p)^{-\frac{1}{2}}$ we can using the substitution method and obtain
  \begin{align*}
      &\int_0^{2\pi} \int_0^{\sqrt{2}} \frac{r}{\left(1-r^2\right)^2-p} \text{d}r \, \text{d}\theta
      = \pi \int_{-1}^1 \frac{1}{{r'}^2-p} \text{d}r'  
      = (-p)^{-\frac{1}{2}} \int_{-(-p)^{-\frac{1}{2}}}^{(-p)^{-\frac{1}{2}}} \frac{1}{{r''}^2+1} \text{d}r''
      = \Theta\left((-p)^{-\frac{1}{2}} \right)\;.
  \end{align*}
  \end{example}

\section{Numerical simulations} \label{sec:num}
In this section, we numerically investigate the analytical results from Section \ref{sec:1-dim} and Section \ref{sec:N-dim} to gain more insight and also to obtain an outlook, on how they can be relevant in a more applied setting. The numerical methods used and discussed follow the theory of \cite{HI} and \cite{LO}.\\
We start by simulating the results obtained in Theorem \ref{thm:1-dim_tool} for the one-dimensional case and the tool function $f_{\alpha}(x) = - |x|^{\alpha}$ on an interval for $\alpha>0$. In order to approximate the solution of the studied SPDE, we use the Euler-Maruyama method, which we introduce considering the differential equation 
\begin{align} \label{eq:pre_num}
    \text{d}u(x,t) = \left(f_\alpha(x)+p\right)\; u(x,t) \, \text{d}t + \sigma \text{d}W(t) \, , \quad u(\cdot,0) = u_0\, , \quad 0 \leq t \leq T \, ,
\end{align}
for $\sigma, T>0>p$ and initial condition $u_0\in L^2(\mathcal{X})$. We then discretize the time interval $[0,T]$ by defining the time step $\delta \mathtt{t} = \frac{T}{\mathtt{nt}}$ for a certain positive integer $\mathtt{nt}$, to be the number of steps in time, and $ \tau_i := i\; \delta \mathtt{t}$, the time passed after $i$ time steps. Further, we discretize also the space interval $\mathcal{X}=[-L, L]$ into its internal points $r_n=-L+2\frac{n}{\mathtt{N}+1}L$ for $n\in\{1,\dots,\mathtt{N}\}$, with $\mathtt{N}$ being an integer that defines the number of mesh points. The numerical approximations of $u(\cdot,\tau_i)$, $f_\alpha$ and $W(\cdot,\tau_i)$ are labeled respectively as $\mathtt{u}_i, \mathtt{f}_\alpha, \mathtt{W}_i \in\mathbb{R}^{\mathtt{N}+2}$ for any $i\in\{0,\dots,\mathtt{nt}\}$. Then, by the implicit Euler-Maruyama method, the numerical simulation takes the form 
\begin{align*}
    \mathtt{u}_i = \mathtt{u}_{i-1} + \left( \mathtt{f}_\alpha+p \right) \mathtt{u}_{i} \delta \mathtt{t} + \sigma (\mathtt{W}_i - \mathtt{W}_{i-1}) \, , \quad i = 1,2, \ldots, \mathtt{nt} 
\end{align*}
which approximates the integral form of the SPDE. For fixed $i\in\{1,\ldots,\mathtt{nt}\}$, the term $\mathtt{W}_i - \mathtt{W}_{i-1}$ can be expressed as 
\begin{align*}
   \mathtt{W}_i - \mathtt{W}_{i-1}= \sqrt{\delta \mathtt{t}} \sum_{m=1}^{M} \sqrt{q_m} W_i^{(m)}\mathtt{b}_m
\end{align*}
with $M\leq \mathtt{N}$ being the number of directions in the space function on which the noise is taken numerically into account, $\left\{W_i^{(m)}\right\}_{m\in\{1,\dots,M\}}$ a collection of independent standard Gaussian random variables and the randomly generated $\left\{\left(q_m, \mathtt{b}_m \right)\right\}_{m\in\{1,\dots,M\}}$, which are respectively the first $M$ eigenvalues and approximations in $\mathbb{R}^{\mathtt{N}+2}$ of the eigenfunctions of the covariance operator $Q$ of $W$ in \eqref{eq:pre_num}.\footnote{The eigenvalues $\left\{\left(q_m \right)\right\}_{m\in\{1,\dots,M\}}$ are generated uniformly in $[0.5,2]$.} \\
We numerically simulate the results from Theorem \ref{thm:1-dim_tool} for the function $f_{\alpha}=- |x|^{\alpha}$ for $\alpha > 0$ as defined by \eqref{eq:exfunc}. We consider the case $g(x)= \mathbbm{1}_{[-\varepsilon,\varepsilon]}$ for $\left|\left\{r_n\in[-\varepsilon,\varepsilon]\;|\; n\in\{1,\dots,\mathtt{N}\}\right\}\right|=M$.\footnote{The orthogonal matrix whose indices are $\left\{\mathtt{b}_m(n) | r_n\in[-\varepsilon,\varepsilon] \right\}_{m\in\{1,\dots,M\}}$ is generated through $\mathtt{O}(M)$ Haar distribution, for $\mathtt{O}(M)$ that indicates the space of orthogonal $M\times M$ matrices with real-valued elements \cite{mezzadri2006generate}. The rest of the elements in $\left\{\left(\mathtt{b}_m \right)\right\}_{m\in\{1,\dots,M\}}$, which are irrelevant to the study of the time-asymptotic variance along $g$, can be obtained through Graham-Schmidt method.} Further, we assume $Q$ and its inverse $Q^{-1}$ to be bounded operators. The fact that $Q$ is not assumed to be the identity operator makes the random variables $u(x_1,t)$ and $u(x_2,t)$ dependent for any $x_1,x_2\in\mathcal{X}$ and $t>0$. Hence the simulation of $u$ must be studied as an SPDE rather than a collection of SDEs on the resolution points.
\begin{figure}[h!]
    \centering
    \begin{overpic}[width= 0.6\textwidth]{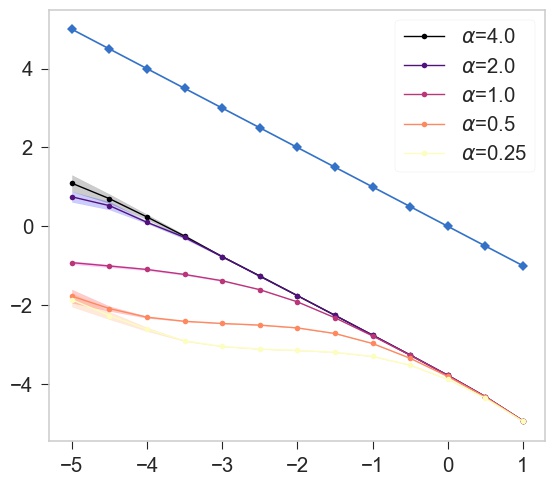}
    \put(460,0){\large{$\log_{10} (-p)$}}
    \put(-30,600){\large{\rotatebox{270}{$\log_{10} \left(\langle V_{\infty}g, g \rangle \right)$}}}
    \end{overpic}
    \caption{Log-log plot that describes the behaviour of $\langle V_{\infty}g, g \rangle $ as $p$ approaches $0^-$ in accordance to the choice of the tool function $f_\alpha$. The circles are obtained as the mean value of $\log_{10} \left(\langle V_{\infty}g, g \rangle \right)$ given by 10 independent simulations. The shaded areas have width equal to the numerical standard deviation. Lastly, the blue line has a slope equal to $-1$ and is a reference for the scaling law. \\
    For $\alpha\geq 1$ the expected slope from Theorem \ref{thm:1-dim_tool} is shown close to $p=10^{-5}$. For $\alpha<1$ the convergence is visible until $p$ assumes small values. In fact, for small $\mathtt{N}$, the log-log plot displays slope $-1$ induced by the divergence being only perceived on $x=0$ and therefore leading to a behaviour similar to that of a Ornstein-Uhlenbeck process \cite{kuehn2015multiple}.}
    \label{fig:loglog}
\end{figure}
We simulate the projection $\langle u(\cdot,\tau_i), g \rangle$ with $\textnormal{proj}_i=\underset{n:r_n\in[-\varepsilon,\varepsilon]}{\sum} \mathtt{u}_i(n)$ for any $i\in\{1,\dots,nt\}$. We approximate the behaviour of the scalar product that defines the variance along $g$, $\langle V_{\infty}g, g \rangle$, as the numerical variance in time $i$ of $\textnormal{proj}_i$. The results are displayed as a log-log plot in Figure \ref{fig:loglog}. As we are interested in the behaviour for $p$ approaching zero from below, we look especially at negative values of $\log_{10}(-p)$. Validating the analytic results in Table \ref{tab:onedim}, for $\alpha > 1$ we observe $\log_{10}( \langle V_{\infty}g, g \rangle)$ to assume a negative slope of $-1 + \frac{1}{\alpha}$ as $\log_{10}(-p)$ approaches $-\infty$. In the case $\alpha = 1$ we expect a logarithmic divergence in the log-log plot as $\log_{10}(-p)$ decreases. Lastly, we expect convergence for $0<\alpha<1$ which is shown up to small values of $p$ due to numerical errors.\footnote{$L=1$, $\mathtt{N}=99999$, $\mathtt{nt}=10000000$, $\delta\mathtt{t}=0.1$, $\sigma=0.1$, $\varepsilon=0.01$, $M=999$. The shape of $Q$ is randomly generated.} \\

Next, we study the two dimensional case and numerically simulate the asymptotic behaviour of the upper bound of $\langle V_{\infty}g,g \rangle$ as described in Theorem \ref{thm:2-dim_analytic}. In particular, our goal is to cross-validate the results analytically found for indices $i_1, i_2$ such that $i_2 > i_1 > 1$. For this case, we have found the scaling law of the upper bound as $\Theta\left((-p)^{-1+\frac{1}{i_2}}\right)$. We therefore compute the dependence of $\int_0^\varepsilon \int_0^\varepsilon \frac{1}{x_1^{i_1}x_2^{i_2}-p}\, \text{d}x_1 \, \text{d}x_2$ on $p>0$ in a log-log plot.

\begin{figure}[h!]
    \centering
    \subfloat[\footnotesize{$\varepsilon=1, i_1 = 2, i_2 = 10$}]{\begin{overpic}[width= 0.43\textwidth]{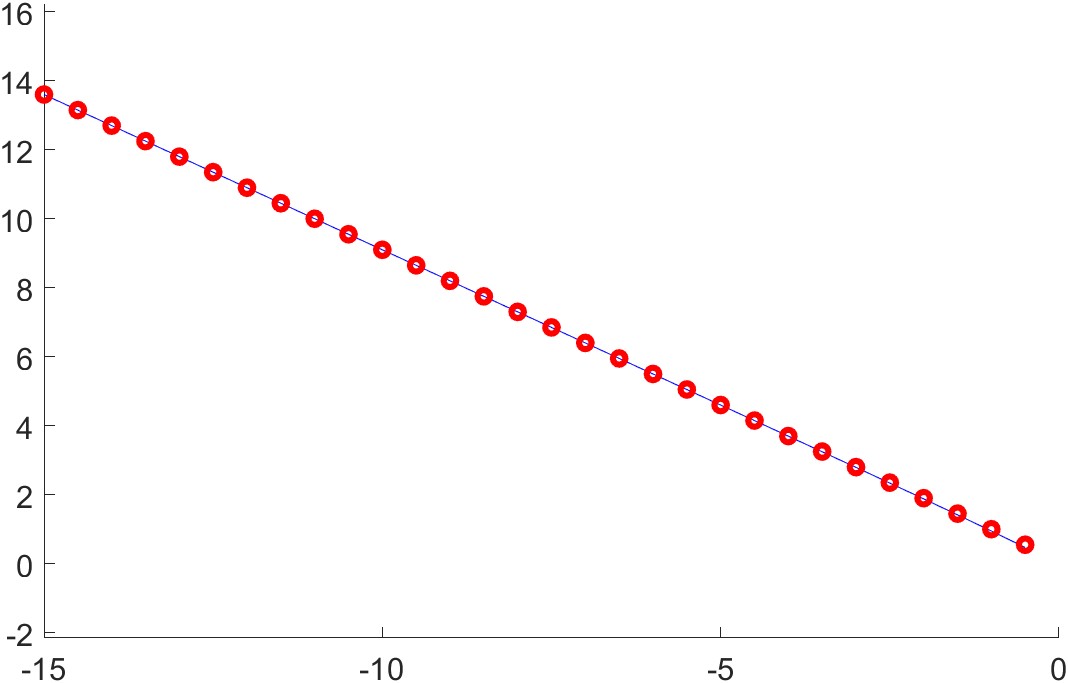}
    \put(440,-30){\footnotesize{$\log_{10}(-p)$}}
    \put(-90,530){\footnotesize{\rotatebox{270}{$\log_{10} \left(\int_0^\varepsilon \int_0^\varepsilon \frac{1}{x^{j_{\ast}} - p}\, \text{d}x\right)$}}}
    \end{overpic}}
    \hspace{1cm}
    \subfloat[\footnotesize{$\varepsilon=1, i_1 = 1, i_2 = 2$}]{\begin{overpic}[width= 0.43\textwidth]{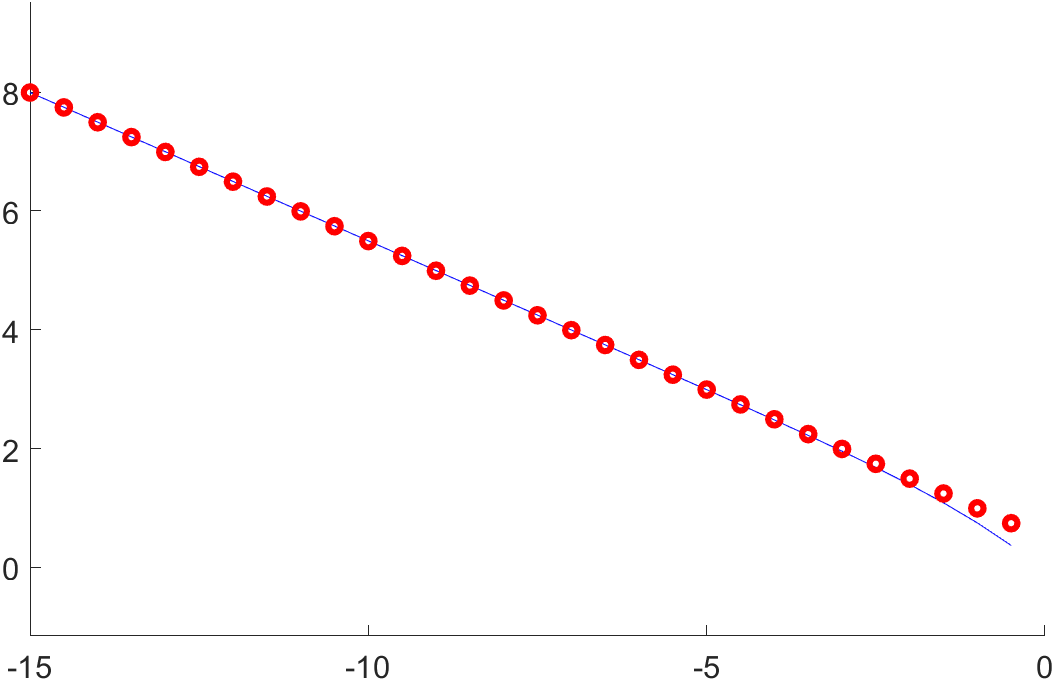}
    \put(440,-30){\footnotesize{$\log_{10}(-p)$}}
    \put(-90,530){\footnotesize{\rotatebox{270}{$\log_{10} \left(\int_0^\varepsilon \int_0^\varepsilon \frac{1}{x^{j_{\ast}} - p}\, \text{d}x\right)$}}}
    \end{overpic}}
    \vfill
    \subfloat[\footnotesize{$\varepsilon=0.1, i_1 = 2, i_2 = 10$}]{\begin{overpic}[width= 0.43\textwidth]{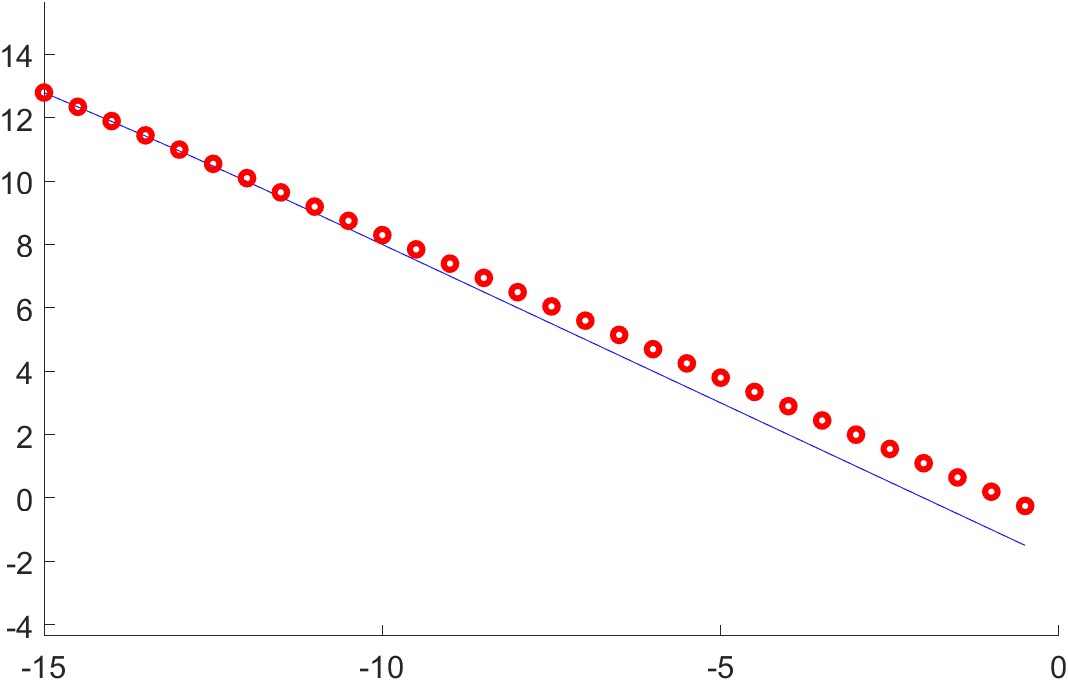}
    \put(440,-30){\footnotesize{$\log_{10}(-p)$}}
    \put(-90,530){\footnotesize{\rotatebox{270}{$\log_{10} \left(\int_0^\varepsilon \int_0^\varepsilon \frac{1}{x^{j_{\ast}} - p}\, \text{d}x\right)$}}}
    \end{overpic}}
    \hspace{1cm}
    \subfloat[\footnotesize{$\varepsilon=0.1, i_1 = 1, i_2 = 2$}]{\begin{overpic}[width= 0.43\textwidth]{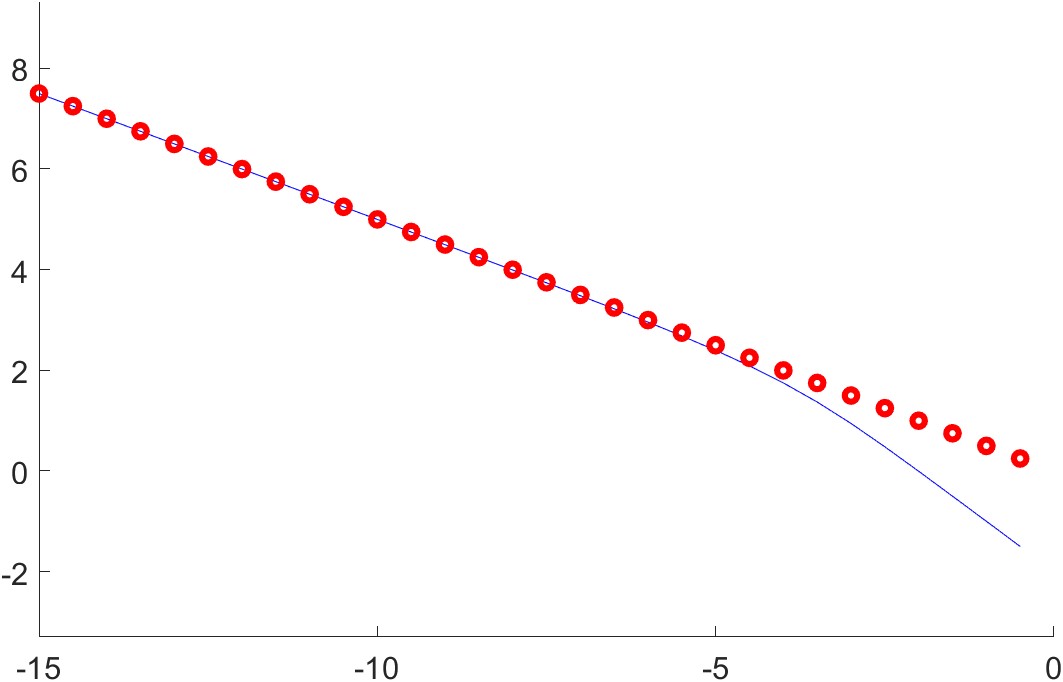}
    \put(440,-30){\footnotesize{$\log_{10}(-p)$}}
    \put(-90,530){\footnotesize{\rotatebox{270}{$\log_{10} \left(\int_0^\varepsilon \int_0^\varepsilon \frac{1}{x^{j_{\ast}} - p}\, \text{d}x\right)$}}}
    \end{overpic}}
    \caption{The solid line describes the scaling law of the upper bound for the two-dimensional problem illustrated by $\log_{10} \left(\int_0^\varepsilon \int_0^\varepsilon \frac{1}{x^{j_{\ast}} - p}\, \text{d}x\right)$  and decreasing $\log_{10}(-p)$ with $x = (x_1, x_2)$ and $j_\ast=(i_1,i_2)$. The circle line serves as comparison with slope $-1+\frac{1}{i_2}$.}
    \label{fig:twodim_numerics}
\end{figure}
The results shown in Figure \ref{fig:twodim_numerics} are obtained for different choices of $\varepsilon$ and of the exponents $i_2$ and $i_1$ such that $i_2 >i_1 > 1$. We see that our results are confirmed as the double integrals, displayed as solid lines, present slope $-1 + \frac{1}{i_2}$ for small values of $-p$. Decreasing the parameter $\varepsilon>0$ implies a decrease of the threshold $q_c>0$ such that for any $-p>q_c$ the slope of $\log_{10} \left(\int_0^\varepsilon \int_0^\varepsilon \frac{1}{x_1^{i_1} x_2^{i_2} - p}\, \text{d}x_1 \text{d}x_2 \right)$ in the log-log plot is approximately $-1$. This is induced by the fact that $x_\ast=0\in\mathbb{R}^2$ is the unique local zero of the continuous function $-x_1^{i_1} x_2^{i_2}$.

\section*{Conclusion and Outlook}
We have studied the time-asymptotic covariance operator of the system \eqref{sist} along certain types of functions $g\in L^2(\mathcal{X})$, i.e., $\langle V_{\infty}g,g \rangle$. Our results contribute to the study of early-warning signs for the qualitative change that arises as the parameter $p<0$ in \eqref{sist} approaches $0$ from below. A critical role for the behaviour of the early-warning sign is given to the values assumed by the non-positive function $f:\mathcal{X}\subseteq\mathbb{R}^N\to\mathbb{R}$, in \eqref{sist}, in the neighbourhood of one of its roots, labeled as $x_\ast$.\\
We have found sharp scaling laws for locally analytic functions $f$ on a one-dimensional domain by utilizing the tool function $f_{\alpha}(x)$. For the two- and three-dimensional cases, we have studied the asymptotic behaviour of upper bounds of $\langle V_{\infty}g,g \rangle$. For general dimensions $N$ we have found convergence for the variance for certain functions $f$. In these cases, the search for a finer early-warning sign that provides more information might be useful. \\
In future work, a novel route could be the study of similar early-warning sign problems for different types of operators, such as linear differential operators that are not self-adjoint. Another promising research direction is to aim to determine weaker assumptions on the linear operators present in \eqref{new_sist} or assuming unbounded $g$ on $x_\ast$ for $N>1$. The results for the one-dimensional and higher-dimensional domain have been proven for analytic functions $f$. A deeper discussion on multiplication operators with functions that do not behave locally like polynomials could be of further interest. Lastly, the results obtained in the paper, as tables listing rates, could be basis for an induction method that gives also results for $N>3$.

\appendix

\section{Appendix: Proof Theorem \ref{thm:2-dim_analytic}} \label{appendA}

The following proof justifies the description of the scaling law of the upper bound presented in Theorem \ref{thm:2-dim_analytic} and its results are summarized in Table \ref{tab:twodim}. The proof is based upon a splitting of the upper bound into two summands, labeled as $\mathfrak{A}$ and $\mathfrak{B}$, the first of which is studied similarly as the methods employed in the proof of Theorem \ref{thm:1-dim_tool}.
\begin{proof}[Proof Theorem \ref{thm:2-dim_analytic}]
Up to permutation of indices, we assume $i_2\geq i_1 \geq 1$ for simplicity. We also set $\varepsilon=1$, up to rescaling of the space variable $x$ and subsequently of $\mathcal{X}$. Lemma \ref{lm:N-dim_up} implies for a positive constant $C>0$ that
\begin{align*} 
    \langle V_{\infty}g,g \rangle 
    &\leq \frac{1}{C} \int_0^1 \int_0^1 \frac{1}{x_1^{i_1}x_2^{i_2}-p} \, \text{d}x_2 \text{d}x_1 
    = \frac{-1}{C p} \mathlarger{\mathlarger{\int}}_0^1 \mathlarger{\mathlarger{\int}}_0^1 \frac{1}{\left(x_1(-p)^{-\frac{1}{2 i_1}}\right)^{i_1} \left(x_2(-p)^{-\frac{1}{2 i_2}} \right)^{i_2}+1} \, \text{d}x_2 \text{d}x_1 \, .
 \end{align*}
For convenience, we will write $q:=-p >0$. Through integration by substitution with $y_n:= x_nq^{-\frac{1}{2 i_n}}$ for $n \in \{ 1,2\}$ we obtain 
\begin{align}  \label{eq:twodimstart}
    \langle V_{\infty}g,g \rangle 
    &\leq \frac{1}{Cq} \mathlarger{\mathlarger{\int}}_0^1 \mathlarger{\mathlarger{\int}}_0^1 \frac{1}{\left(x_1 q^{-\frac{1}{2 i_1}}\right)^{i_1} \left(x_2 q^{-\frac{1}{2 i_2}} \right)^{i_2}+1} \, \text{d}y_2 \text{d}y_1 \\
    &= \frac{1}{C} q^{-1 + \frac{1}{2}\left(\frac{1}{i_1}+\frac{1}{i_2}\right)}\int_0^{ q^{-\frac{1}{2i_1}}}\int_0^{ q^{-\frac{1}{2i_2}}} \frac{1}{y_1^{ i_1}y_2^{i_2} +1} \, \text{d}y_2 \, \text{d} y_1 \,. \nonumber
 \end{align}
We evaluate the integral by substituting $z = \frac{1}{y_1^{i_1}y_2^{i_2} + 1}$  to get
\begin{align}  \label{eq:twodim}
    \langle V_{\infty}g,g \rangle &\leq \frac{1}{C} q^{-1 + \frac{1}{2}\left(\frac{1}{i_1}+\frac{1}{i_2}\right)}\int_0^{ q^{-\frac{1}{2i_1}}}\int_0^{ q^{-\frac{1}{2i_2}}} \frac{1}{y_1^{ i_1}y_2^{i_2} +1} \, \text{d}y_2 \, \text{d} y_1  \nonumber \\
    &= \frac{1}{C} q^{-1 + \frac{1}{2}\left(\frac{1}{i_1}+\frac{1}{i_2}\right)} \mathlarger{\mathlarger{\int}}_0^{ q^{-\frac{1}{2i_1}}}\int_{ \frac{1}{y_1^{i_1} q^{-\frac{1}{2}}+1}}^{1} z z^{-2} \frac{1}{i_2} \left(\frac{1}{z}-1\right)^{\frac{1}{i_2} -1}y_1^{-\frac{i_1}{i_2}} \, \text{d}z \, \text{d}y_1\\
    &= \frac{1}{C} q^{-1 + \frac{1}{2}\left(\frac{1}{i_1}+\frac{1}{i_2}\right)} \mathlarger{\mathlarger{\int}}_0^{ q^{-\frac{1}{2i_1}}} \left( \frac{ q^{-\frac{1}{2i_2}}}{y_1^{i_1} q^{-\frac{1}{2}}+1}+ \int_{\frac{1}{y_1^{i_1} q^{-\frac{1}{2}}+1}}^1 \left( \frac{1}{z} - 1 \right)^{\frac{1}{i_2}}y_1^{-\frac{i_1}{i_2}}\, \text{d}z \right) \, \text{d}y_1  \nonumber \\
    &= \frac{1}{C} \underbrace{ q^{-1 + \frac{1}{2 i_1}}\int_0^{ q^{-\frac{1}{2i_1}}} \frac{1}{y_1^{i_1} q^{-\frac{1}{2}}+1} \, \text{d}y_1}_{\mathfrak{A}} 
    + \frac{1}{C} \underbrace{q^{-1 + \frac{1}{2}\left(\frac{1}{i_1}+\frac{1}{i_2}\right)} \mathlarger{\mathlarger{\int}}_0^{ q^{-\frac{1}{2i_1}}}\int_{\frac{1}{y_1^{i_1} q^{-\frac{1}{2}}+1}}^1 \left( \frac{1}{z} - 1 \right)^{\frac{1}{i_2}}y_1^{-\frac{i_1}{i_2}}\, \text{d}z \, \text{d}y_1}_{\mathfrak{B}} \, . \nonumber
 \end{align}
In the following, we first analyze $\mathfrak{A}$, before evaluating the order of $\mathfrak{B}$. Further, as seen above, the choice of $i_1$ and $i_2$ dictates the rate of the upper bound. Hence, we look at different cases of choosing $i_1$ and $i_2$, whose numbering is displayed in Table \ref{tab:twodim}.
\begin{itemize}[label=$\lozenge$]
\item Cases 1 and 3 for $\mathbf{\mathfrak{A}}$.\\
Note that the order of $\mathfrak{A}$ is independent of $i_2$, which simplifies the study of its scaling law. We start by discussing $\mathfrak{A}$ for $i_2 \geq i_1 > 1$. Through the substitution $y_1' = y_1 q^{-\frac{1}{2i_1}}$ it follows that 
\begin{align*} 
    \mathfrak{A} = 
    q^{-1 + \frac{1}{i_1}}\int_0^{q^{-\frac{1}{i_1}}} \frac{1}{y_1'^{i_1} +1} \, \text{d}y_1' \; .
 \end{align*} 
As already established in \eqref{eq:onedim_conv}, the integral is converging for $p \to 0^{-}$, and therefore we find a rate of divergence for the first summand given by
\begin{align} \label{eq:mathfrakA_1}
\Theta \left( (-p)^{-1 + \frac{1}{i_1}} \right)
\end{align}
for $p \to 0^{-}$ and $i_2 \geq i_1 > 1$. 
\item Cases 2 and 4 for $\mathbf{\mathfrak{A}}$.\\
Next, we consider $\mathfrak{A}$ assuming that $i_2 \geq i_1 = 1$. Through the same substitutions present in the previous case we get  
\begin{align*} 
    \mathfrak{A} 
    = \int_0^{q^{-1}} \frac{1}{y_1' +1} \, \text{d}y_1'
    =  \log\left(  q^{-1} +1 \right) \;.
 \end{align*}
This expression diverges for $q$ approaching zero from above, and for $i_2 \geq i_1 = 1$ we see that
\begin{align} \label{eq:mathfrakA_2}
    \Theta \left(\log \left((-p)^{-1}\right) \right) 
    = \Theta \left(-\log \left(-p\right) \right)
\end{align}
is the rate of divergence for $p$ approaching zero from below.  
\end{itemize}
\begin{itemize}[label=$\blacklozenge$]
\item Case 1 for $\mathbf{\mathfrak{B}}$.\\
We consider the case $i_2 > i_1 > 1$. In regard to $\mathfrak{B}$, our first goal is to switch integrals for more convenient computation.
Through Fubini's Theorem we obtain 
\begin{align}  \label{eq:case1_1}
    \mathfrak{B} &= q^{-1 + \frac{1}{2}\left(\frac{1}{i_1}+\frac{1}{i_2}\right)} \int_{\frac{1}{ q^{-1}+1}}^1 \int_{\left(\frac{1}{z} -1\right)^{\frac{1}{i_1}} q^{\frac{1}{2i_1}}}^{ q^{-\frac{1}{2i_1}}} \left( \frac{1}{z} - 1 \right)^{\frac{1}{i_2}} y_1^{-\frac{i_1}{i_2}}\, \text{d}y_1 \, \text{d}z \nonumber\\
    &= q^{-1 + \frac{1}{2}\left(\frac{1}{i_1}+\frac{1}{i_2}\right)} \mathlarger{\mathlarger{\int}}_{\frac{1}{ q^{-1}+1}}^1 \left( \frac{1}{z} - 1 \right)^{\frac{1}{i_2}} \left[ \frac{1}{\left(1- \frac{i_1}{i_2}\right)} y_1^{-\frac{i_1}{i_2}+1} \right]_{\left( \frac{1}{z} -1\right)^{\frac{1}{i_1}} q^{\frac{1}{2i_1}}}^{ q^{-\frac{1}{2i_1}}} \, \text{d}z   \\
    & = \frac{q^{-1 + \frac{1}{2}\left(\frac{1}{i_1}+\frac{1}{i_2}\right)}}{1-\frac{i_1}{i_2}} \int_{\frac{1}{ q^{-1}+1}}^1 \left( \frac{1}{z} - 1 \right)^{\frac{1}{i_2}} \left( q^{\frac{1}{2i_2}-\frac{1}{2i_1}} - \left( \frac{1}{z}-1 \right)^{\frac{1}{i_1} - \frac{1}{i_2}} q^{\frac{1}{2i_1}-\frac{1}{2i_2}} \right)\, \text{d}z \nonumber \\
    & = \Big(1-\frac{i_1}{i_2}\Big)^{-1}
    \Bigg(  q^{-1 + \frac{1}{i_2}} \int_{\frac{1}{ q^{-1}+1}}^1 \left( \frac{1}{z} - 1 \right)^{\frac{1}{i_2}} \, \text{d}z 
    - q^{-1 + \frac{1}{i_1}} \int_{\frac{1}{ q^{-1}+1}}^1 \left( \frac{1}{z}-1 \right)^\frac{1}{i_1} \, \text{d}z \Bigg) 
     \; . \nonumber
 \end{align}
By substitution with $z^{\prime} = \frac{1}{z}-1$ we obtain 
\begin{align} \label{eq:case1_2}
    & \mathfrak{B}
    = \Big(1-\frac{i_1}{i_2}\Big)^{-1}
    \Bigg( q^{-1 + \frac{1}{i_2}} \int_0^{ q^{-1}} \frac{z^{\prime \frac{1}{i_2}}}{(z^{\prime}+1)^2} \, \text{d}z' 
    - q^{-1 + \frac{1}{i_1}} \int_0^{q^{-1}} \frac{z^{\prime \frac{1}{i_1}}}{(z^{\prime}+1)^2} \, \text{d}z' \Bigg) \;.
 \end{align}
Given $i_2>i_1>1$, the integrals in \eqref{eq:case1_2} can be uniformly bounded for any $q>0$. Hence we find that for $\mathfrak{B}$ the rate of divergence is given by 
\begin{align*} 
    \Theta \left( q^{-1+\frac{1}{i_2}} \right) - \Theta \left( q^{-1+ \frac{1}{i_1}}\right) =   \Theta \left( (-p)^{-1+\frac{1}{i_2}} \right) - \Theta \left( (-p)^{-1+ \frac{1}{i_1}}\right) = \Theta\left( (-p)^{-1+\frac{1}{i_2}} \right)
 \end{align*}
as $p$ approaches zero from below. From \eqref{eq:twodim} we get the rate of divergence of the upper bound as
\begin{align*} 
    \Theta\left((-p)^{-1+ \frac{1}{i_1}}\right)
    + \Theta\left((-p)^{-1+\frac{1}{i_2}}\right) = \Theta\left((-p)^{-1+ \frac{1}{i_2}}\right)\, .
 \end{align*}
for $i_2 > i_1 > 1$.
\item Case 2 for $\mathbf{\mathfrak{B}}$.\\
We suppose now that $i_2 > i_1 =1.$ We can follow the steps in the case above, up to \eqref{eq:case1_2}. We then note that 
\begin{align}  \label{integral_tool}
     & \int_0^{ q^{-1}} \frac{z'}{(z^{\prime}+1)^2} \, \text{d}z'
     =\log\left( q^{-1}+1\right) + \frac{1}{q^{-1}+1}-1 \;.
 \end{align}
Therefore, we obtain that $\mathfrak{B}$ diverges with rate 
\begin{align*} 
    \Theta\left( (-p)^{-1 + \frac{1}{i_2}} \right) - \Theta \left( \log ((-p)^{-1}) \right) = \Theta\left( (-p)^{-1 + \frac{1}{i_2}} \right) \quad ,
 \end{align*} 
for $p \to 0^{-}$. Combining this result with equation (\ref{eq:mathfrakA_2}) we obtain for $i_2 > i_1 =1$ an overall rate of divergence of 
\begin{align*}
    \Theta\left(- \log (-p) \right)
   + \Theta\left( (-p)^{-1+\frac{1}{i_2}} \right)   
   = \Theta\left( (-p)^{- 1 + \frac{1}{i_2}} \right)\,.
\end{align*}
\item Case 3 for $\mathbf{\mathfrak{B}}$.\\
For the third case we consider $\mathfrak{B}$ with $i_2 = i_1 = i > 1$, for which we get
\begin{align*} 
     \mathfrak{B} &= q^{-1 + \frac{1}{i}} \int_{\frac{1}{ q^{-1}+1}}^1 \int_{\left(\frac{1}{z} -1\right)^{\frac{1}{i}}  q^{\frac{1}{2i}}}^{ q^{-\frac{1}{2i}}} \left( \frac{1}{z} - 1 \right)^{\frac{1}{i}} y_1^{-1}\, \text{d}y_1 \, \text{d}z  \\
    &= q^{-1 + \frac{1}{i}} \int_{\frac{1}{ q^{-1}+1}}^1 \left( \frac{1}{z} - 1 \right)^{\frac{1}{i}} \left( \log \left( q^{-\frac{1}{2i}} \right) - \log \left(\left(\frac{1}{z} -1 \right)^{\frac{1}{i}} q^{\frac{1}{2i}} \right) \right)\, \text{d}z \\
    &= q^{-1 + \frac{1}{i}} 
    \left( - \frac{\log(q)}{i}
    \int_{\frac{1}{ q^{-1}+1}}^1 \left( \frac{1}{z} - 1 \right)^{\frac{1}{i}} \, \text{d}z 
    - \frac{1}{i} \int_{\frac{1}{ q^{-1}+1}}^1 \left( \frac{1}{z} - 1 \right)^{\frac{1}{i}} \log\left( \frac{1}{z} - 1 \right)\, \text{d}z\right) \; .
 \end{align*}
Again, we substitute $z^{\prime} = \frac{1}{z}-1$ to get 
\begin{align} \label{eq:case3}
     \mathfrak{B} = q^{-1 + \frac{1}{i}} 
    \left(  - \frac{\log(q)}{i} 
    \int_0^{ q^{-1}} \frac{z'^{\frac{1}{i}}}{(z'+1)^2} \, \text{d}z' 
    - \frac{1}{i} \int_0^{ q^{-1}} \frac{z'^{\frac{1}{i}} \log(z')}{(z'+1)^2} \, \text{d}z'\right) \, .
 \end{align}
The integrals in \eqref{eq:case3} are uniformly bounded for any $q>0$ due to $i>1$. Equation \eqref{eq:case3} implies that $\mathfrak{B}$ assumes, in this case, rate of divergence
\begin{align*} 
    \Theta \left( -(-p)^{-1+ \frac{1}{i}}\log(-p) \right) +  \Theta \left( (-p)^{-1 + \frac{1}{i}}\right) =  \Theta \left( -(-p)^{-1+ \frac{1}{i}}\log(-p) \right) 
 \end{align*} 
for $ p \to 0^{-}$.\\
From \eqref{eq:twodim} and \eqref{eq:mathfrakA_1} we get that the rate of divergence of the upper bound is
\begin{align*} 
    \Theta \left((-p)^{-1+\frac{1}{i}} \right) 
    +  \Theta \left( -(-p)^{-1+\frac{1}{i}}\log(-p) \right)
    =  \Theta \left( - (-p)^{-1+\frac{1}{i}}\log(-p) \right)\, .
 \end{align*}
\item Case 4 for $\mathbf{\mathfrak{B}}$.\\
In the last case we suppose that $i_2 = i_1 =i=1$. Hence \eqref{eq:case3} assumes the form
\begin{align}  \label{eq:twodim_case4}
     \mathfrak{B} &=
     - \log(q) 
    \int_0^{ q^{-1}} \frac{z'}{(z'+1)^2} \, \text{d}z' 
    - \int_0^{ q^{-1}} \frac{z' \log(z')}{(z'+1)^2} \, \text{d}z'\\
    &=  - \log(q) 
    \Bigg(\log( q^{-1}+1)+\frac{1}{ q^{-1}+1}-1\Bigg)
    - \int_0^{ q^{-1}} \frac{z' \log(z')}{(z'+1)^2} \, \text{d}z' \;. \nonumber
 \end{align}
We obtain therefore from \eqref{eq:twodim_case4} and \eqref{nice2}, in Appendix \ref{AppendC}, that the two leading terms in $\mathfrak{B}$ diverge respectively as $\log^2(q)$ and as $-\frac{1}{2}\log^2(q)$. Hence, we get divergence of $\mathfrak{B}$ with the rate
\begin{align*} 
    \Theta \left( \log^2(-p) \right) \quad \text{for } p \to 0^{-} \, . 
 \end{align*}
Combining such result with the rate of divergence of \eqref{eq:mathfrakA_2} for $\mathfrak{A}$ we obtain
\begin{align*} 
    \Theta \left(- \log(-p) \right)
    + \Theta \left( \log^2(-p) \right) 
    = \Theta \left( \log^2(-p) \right) \quad \text{for } p \to 0^{-} \;,
 \end{align*}
as an overall upper bound for $i_2 = i_1 = 1$. \\
\end{itemize}
    
\end{proof}

\section{Appendix: Proof Theorem \ref{thm:3-dim_analytic}} \label{appendB}

The subsequent proof validates the description of the rate of divergence of the upper bound discussed in Theorem \ref{thm:3-dim_analytic} and in Table \ref{tab:threedim}. The approach is similar to the one involved in the proof of Theorem \ref{thm:2-dim_analytic} as the upper bound is split in two summands, labeled $\mathfrak{C}$ and $\mathfrak{D}$. The first summand is studied similarly to the upper bound in the stated proof for the two-dimensional case.
\begin{proof}[Proof Theorem \ref{thm:3-dim_analytic}]

Lemma \ref{lm:N-dim_up} implies that
\begin{align*} 
    \langle V_{\infty}g, g \rangle 
    \leq \frac{\sigma^2}{2C} \int_{0}^\varepsilon \int_{0}^\varepsilon \int_{0}^\varepsilon \frac{1}{x_1^{i_1} x_2^{i_2} x_3^{i_3}-p} \, \text{d}x_3 \, \text{d}x_2 \, \text{d}x_1 
 \end{align*}
for a constant $C>0$ and any $p<0$. As in the two-dimensional case, we make use of $q:=-p > 0$ and we fix $\varepsilon=1$, up to rescaling of the spatial variable $x$. Up to permutation of the indices we consider $i_3\geq i_2 \geq i_1 \geq 1$, thus excluding the cases described in Remark \ref{rmk:dim_reduction}.\\

We follow a similar computation of the integral as in Theorem \ref{thm:2-dim_analytic}. First, we study it in the coordinates ${y_n = x_nq^{-\frac{1}{3i_n}}}$ for all $n \in \{1,2,3 \}$ and then we substitute $y_3$ with ${z = \frac{1}{y_1^{i_1}y_2^{i_2}y_3^{i_3}+1}}$ to obtain
\begin{align*}  
    & \int_{0}^1 
    \int_{0}^1 
    \int_{0}^1 
    \frac{1}{ x_1^{i_1} x_2^{i_2} x_3^{i_3}+q} \, \text{d}x_3 \, \text{d}x_2 \, \text{d}x_1 \\
    &=q^{-1 + \frac{1}{3}\left( \frac{1}{i_1}+\frac{1}{i_2}+\frac{1}{i_3} \right)} 
    \int_0^{ q^{-\frac{1}{3i_1}}} 
    \int_0^{ q^{-\frac{1}{3i_2}}}
    \int_0^{ q^{-\frac{1}{3i_3}}} 
    \frac{1}{y_1^{i_1}y_2^{i_2}y_3^{i_3}+1} \, \text{d}y_3 \, \text{d}y_2 \, \text{d}y_1 \\
    &= q^{-1 + \frac{1}{3} \left( \frac{1}{i_1}+\frac{1}{i_2}+\frac{1}{i_3} \right)} 
    \mathlarger{\mathlarger{\int}}_0^{ q^{-\frac{1}{3i_1}}} 
    \mathlarger{\mathlarger{\int}}_0^{ q^{-\frac{1}{3i_2}}} 
    \int_{\frac{1}{y_1^{i_1}y_2^{i_2} q^{-\frac{1}{3}}+1}}^1 z z^{-2} \frac{1}{i_3}\left( \frac{1}{z}-1\right)^{\frac{1}{i_3}-1}y_1^{-\frac{i_1}{i_3}}y_2^{-\frac{i_2}{i_3}} \, \text{d}z \, \text{d}y_2 \, \text{d}y_1 \, .
 \end{align*}
Through integration by parts on the inner integral it follows that
\begin{align}
     \label{eq:threedim}
    & \int_{0}^1 
    \int_{0}^1 
    \int_{0}^1 
    \frac{1}{ x_1^{i_1} x_2^{i_2} x_3^{i_3}+q} \, \text{d}x_3 \, \text{d}x_2 \, \text{d}x_1 \nonumber \\
    &=\underbrace{ q^{-1 + \frac{1}{3}\left(\frac{1}{i_1}+\frac{1}{i_2}\right)} 
    \int_0^{ q^{-\frac{1}{3i_1}}} 
    \int_0^{ q^{-\frac{1}{3i_2}}} 
    \frac{1}{y_1^{i_1}y_2^{i_2} q^{-\frac{1}{3}}+1} \, \text{d}y_2 \, \text{d}y_1}_{\mathfrak{C}}\\
    &+ \underbrace{q^{-1 + \frac{1}{3} \left( \frac{1}{i_1}+\frac{1}{i_2}+\frac{1}{i_3} \right) }
    \mathlarger{\mathlarger{\int}}_0^{ q^{-\frac{1}{3i_1}}} 
    \mathlarger{\mathlarger{\int}}_0^{ q^{-\frac{1}{3i_2}}}
    \int_{\frac{1}{y_1^{i_1}y_2^{i_2} q^{- \frac{1}{3}}+1}}^{1}
    \left( \frac{1}{z} - 1\right)^{\frac{1}{i_3}} y_1^{-\frac{i_1}{i_3}}y_2^{-\frac{i_2}{i_3}} \, \text{d}z \, \text{d}y_2 \, \text{d}y_1}_{\mathfrak{D}} \,. \nonumber
\end{align}
As in Theorem \ref{thm:2-dim_analytic}, we study each of the summands independently. Further, we distinguish between different choices of values for $i_n$ for $n \in \{ 1,2,3 \}$. For each case, we find the orders of $\mathfrak{C}$, $\mathfrak{D}$ and the overall upper bound. The numbering of the cases is reported in Table \ref{tab:threedim}.

\begin{itemize}[label=$\lozenge$]
\item Cases 1 - 8 for $\mathbf{\mathfrak{C}}$.\\
We assume $i_1,i_2,i_3>0$. We change the variables in the integral with $y_n^{\prime}=y_n q^{-\frac{1}{6i_n}}$ for $n \in \{ 1 , 2 \}$ to have
\begin{align*} 
    & q^{-1 + \frac{1}{3} \left( \frac{1}{i_1}+\frac{1}{i_2} \right) } 
    \mathlarger{\mathlarger{\int}}_0^{ q^{-\frac{1}{3i_1}}} 
    \mathlarger{\mathlarger{\int}}_0^{ q^{-\frac{1}{3i_2}}} 
    \frac{1}{\left(y_1 q^{- \frac{1}{6 i_1}} \right)^{i_1} \left(y_2 q^{- \frac{1}{6 i_2}} \right)^{i_2}+1}\, \text{d}y_2 \, \text{d}y_1 \\
    &= q^{-1 + \frac{1}{2} \left(\frac{1}{i_1}+ \frac{1}{i_2} \right)}
    \int_0^{ q^{-\frac{1}{2i_1}}}
    \int_0^{ q^{-\frac{1}{2i_2}}}
    \frac{1}{y_1^{\prime i_1}y_2^{\prime i_2} + 1} \, .
 \end{align*}
We note that this expression is equivalent to the integral given by \eqref{eq:twodimstart}, as $i_3$ does not affect such rate of divergence. Hence, we find divergence of $\mathfrak{C}$ for different choices of $i_1, i_2$ with the rate represented in Table \ref{tab:twodim}.
\end{itemize}
\begin{itemize}[label=$\blacklozenge$]
\item Case 1 for $\mathbf{\mathfrak{D}}$.\\
Next, we continue analyzing the component $\mathfrak{D}$ of \eqref{eq:threedim}. First, we assume that $i_3 > i_2 > i_1 > 1$.\\
We change the order of integrals, through Fubini's Theorem, placing the integral on $z$ in the outer position and obtain

\begin{align}  \label{eq:threedim_case8}
    & q^{-1 + \frac{1}{3} \left( \frac{1}{i_1}+\frac{1}{i_2}+\frac{1}{i_3} \right) }
    \mathlarger{\mathlarger{\int}}_0^{ q^{-\frac{1}{3i_1}}} 
    \mathlarger{\mathlarger{\int}}_0^{ q^{-\frac{1}{3i_2}}}
    \int_{\frac{1}{y_1^{i_1}y_2^{i_2} q^{- \frac{1}{3}}+1}}^{1}
    \left( \frac{1}{z} - 1\right)^{\frac{1}{i_3}} y_1^{-\frac{i_1}{i_3}}y_2^{-\frac{i_2}{i_3}} \, \text{d}z \, \text{d}y_2 \, \text{d}y_1\\
    &= q^{-1 + \frac{1}{3} \left( \frac{1}{i_1}+\frac{1}{i_2}+\frac{1}{i_3} \right)} 
    \int_{\frac{1}{ q^{-1}+1}}^1 \int_{\left( \frac{1}{z} - 1\right)^{\frac{1}{i_1}} q^{\frac{2}{3i_1}}}^{ q^{- \frac{1}{3i_1}}}
    \int_{\left( \frac{1}{z} - 1\right)^{\frac{1}{i_2}} q^{\frac{1}{3i_2}}y_1^{-\frac{i_1}{i_2}}}^{ q^{-\frac{1}{3i_2}}} 
    \left( \frac{1}{z} - 1\right)^{\frac{1}{i_3}} y_1^{-\frac{i_1}{i_3}}y_2^{-\frac{i_2}{i_3}} \, \text{d}y_2  \, \text{d}y_1 \, \text{d}z \,. \nonumber
 \end{align}
Since $i_3>i_2$, it follows that $\mathfrak{D}$ is equal to
\begin{align}  \label{threedim_case1}
    &\Big(-\frac{i_2}{i_3}+1\Big)^{-1}  q^{-1 + \frac{1}{3} \left( \frac{1}{i_1}+\frac{1}{i_2}+\frac{1}{i_3} \right)} 
    \int_{\frac{1}{ q^{-1}+1}}^1 \int_{\left( \frac{1}{z} - 1\right)^{\frac{1}{i_1}} q^{\frac{2}{3i_1}}}^{ q^{- \frac{1}{3i_1}}} 
    \left( \frac{1}{z} - 1\right)^{\frac{1}{i_3}} y_1^{-\frac{i_1}{i_3}} 
    \left[y_2^{-\frac{i_2}{i_3}+1} \right]_{y_2=\left( \frac{1}{z} - 1\right)^{\frac{1}{i_2}} q^{\frac{1}{3i_2}}y_1^{-\frac{i_1}{i_2}}}^{ q^{-\frac{1}{3i_2}}}  
    \, \text{d}y_1 \, \text{d}z \nonumber \\
    &=\Big(-\frac{i_2}{i_3}+1\Big)^{-1} \underbrace{ q^{-1 + \frac{1}{3i_1} + \frac{2}{3i_3}} \int_{\frac{1}{ q^{-1}+1}}^1
    \left( \frac{1}{z} - 1\right)^{\frac{1}{i_3}} 
    \int_{\left( \frac{1}{z} - 1\right)^{\frac{1}{i_1}}  q^{\frac{2}{3i_1}}}^{ q^{- \frac{1}{3i_1}}}
    y_1^{- \frac{i_1}{i_3}} \, \text{d}y_1 \, \text{d}z}_{\mathfrak{D}.\text{\RomanNumeralCaps{1}}} \\
    &- \Big(-\frac{i_2}{i_3}+1\Big)^{-1}  \underbrace{ q^{-1 + \frac{1}{3i_1} + \frac{2}{3i_2}} \int_{\frac{1}{ q^{-1}+1}}^1
    \left( \frac{1}{z} - 1\right)^{\frac{1}{i_2}} 
    \int_{\left( \frac{1}{z} - 1\right)^{\frac{1}{i_1}} q^{\frac{2}{3i_1}}}^{q^{- \frac{1}{3i_1}}}
    y_1^{- \frac{i_1}{i_2}} \, \text{d}y_1 \, \text{d}z}_{\mathfrak{D}.\text{\RomanNumeralCaps{2}}} \, . \nonumber 
 \end{align}
From $i_3>i_1$ we know that $\mathfrak{D}.\text{\RomanNumeralCaps{1}}$ is equal to
\begin{align} \label{eq:D.1}
    & \Big(-\frac{i_1}{i_3}+1\Big)^{-1} 
    \Bigg( q^{-1+\frac{1}{i_3}} \int_{\frac{1}{ q^{-1}+1}}^1
    \left( \frac{1}{z} - 1\right)^{\frac{1}{i_3}}  \, \text{d}z 
    - q^{-1 + \frac{1}{i_1}} \int_{\frac{1}{ q^{-1}+1}}^1
    \left( \frac{1}{z} - 1\right)^{\frac{1}{i_1}} \, \text{d}z\Bigg)\\
    &= \Big(-\frac{i_1}{i_3}+1\Big)^{-1} 
    \Bigg( \underbrace{ q^{-1+\frac{1}{i_3}} \int_0^{ q^{-1}}
    \frac{z'^{\frac{1}{i_3}}}{(z'+1)^2}  \, \text{d}z'}_{\mathfrak{D}.\text{\RomanNumeralCaps{1}}.1}
    - \underbrace{ q^{-1 + \frac{1}{i_1}} \int_0^{ q^{-1}}
    \frac{z'^{\frac{1}{i_1}}}{(z'+1)^2} \, \text{d}z'}_{\mathfrak{D}.\text{\RomanNumeralCaps{1}}.2}
    \Bigg)
     \quad , \nonumber
 \end{align}
with $z^{\prime} = \frac{1}{z} -1$. Equivalently, since $i_2>i_1$, the summand $\mathfrak{D}.\text{\RomanNumeralCaps{2}}$ is equal to
\begin{align*}
    & \Big(-\frac{i_1}{i_2}+1\Big)^{-1} \Bigg( \underbrace{ q^{-1+\frac{1}{i_2}} \int_0^{ q^{-1}}
    \frac{z'^{\frac{1}{i_2}}}{(z'+1)^2}  \, \text{d}z'}_{\mathfrak{D}.\text{\RomanNumeralCaps{2}}.1} 
    - \underbrace{ q^{-1 + \frac{1}{i_1}} \int_0^{ q^{-1}}
    \frac{z'^{\frac{1}{i_1}}}{(z'+1)^2} \text{d}z'}_{\mathfrak{D}.\text{\RomanNumeralCaps{2}}.2} \Bigg)
     \, .
 \end{align*}
For $i_1,i_2,i_3>1$, the integrals in $\mathfrak{D}.\text{\RomanNumeralCaps{1}}.1$, $\mathfrak{D}.\text{\RomanNumeralCaps{1}}.2$, $\mathfrak{D}.\text{\RomanNumeralCaps{2}}.1$ and $\mathfrak{D}.\text{\RomanNumeralCaps{2}}.2$ are uniformly bounded for any $q>0$. Hence, the rate of divergence of $\mathfrak{D}$ is
\begin{align*} 
    \Theta\left( (-p)^{-1 + \frac{1}{i_3}} \right)
    - \Theta\left( (-p)^{-1 + \frac{1}{i_1}}\right) - \Theta\left( (-p)^{-1 + \frac{1}{i_2}}\right) + \Theta\left( (-p)^{-1 + \frac{1}{i_1}}\right)  = \Theta\left( (-p)^{-1 + \frac{1}{i_3}} \right) \, .
 \end{align*}
Combining $\mathfrak{C}$ and $\mathfrak{D}$ we get divergence for $i_3 > i_2 > i_1 > 1$ with the upper bound
\begin{align*} 
     \Theta\left( (-p)^{-1+ \frac{1}{i_2}} \right) + \Theta\left( (-p)^{-1 + \frac{1}{i_3}} \right) = \Theta\left( (-p)^{-1 + \frac{1}{i_3}} \right) \,.
 \end{align*}
\item Case 2 for $\mathbf{\mathfrak{D}}$.\\
Next, we consider the case $i_3 > i_2 = i_1 = i > 1$. We can follow the same argumentation of the previous case, until \eqref{eq:D.1}, and prove that 

\begin{align*}
\mathfrak{D}.\text{\RomanNumeralCaps{1}}
=\Big(-\frac{i}{i_3}+1\Big)^{-1} \Big( \mathfrak{D}.\text{\RomanNumeralCaps{1}}.1-\mathfrak{D}.\text{\RomanNumeralCaps{1}}.2 \Big) \;.
\end{align*}
In this case summand $\mathfrak{D}.\text{\RomanNumeralCaps{2}}$ is equal to
\begin{align} \label{case2_3dim}
    & q^{-1 + \frac{1}{i}} \int_{\frac{1}{ q^{-1}+1}}^1
    \left( \frac{1}{z} - 1\right)^{\frac{1}{i}} 
    \int_{\left( \frac{1}{z} - 1\right)^{\frac{1}{i}} q^{\frac{2}{3i}}}^{ q^{- \frac{1}{3i}}}
    y_1^{- 1} \, \text{d}y_1 \, \text{d}z \nonumber \\
    &= - \frac{1}{i} q^{-1 + \frac{1}{i}}
    \int_{\frac{1}{ q^{-1}+1}}^1
    \left( \frac{1}{z} - 1\right)^{\frac{1}{i}} 
    \Bigg( \log(q) 
    + \log\left( \frac{1}{z} - 1\right)
    \Bigg)
    \text{d}z \\
    &= -\frac{1}{i} \underbrace{q^{-1 + \frac{1}{i}}
    \log(q)
    \int_0^{ q^{-1}}
    \frac{z'^{\frac{1}{i}}}{(z'+1)^2} \text{d}z}_{\mathfrak{D}.\text{\RomanNumeralCaps{2}}.3}
    -\frac{1}{i} \underbrace{q^{-1 + \frac{1}{i}} \int_0^{ q^{-1}}
    \frac{z'^{\frac{1}{i}} \log(z')}{(z'+1)^2} \text{d}z}_{\mathfrak{D}.\text{\RomanNumeralCaps{2}}.4} \;. \nonumber
 \end{align}
The integrals in $\mathfrak{D}.\text{\RomanNumeralCaps{2}}.3$ and $\mathfrak{D}.\text{\RomanNumeralCaps{2}}.4$ are uniformly bounded for any $q>0$, therefore the rate of divergence of $\mathfrak{D}.\text{\RomanNumeralCaps{2}}$ is $\Theta\left( -(-p)^{1+\frac{1}{i}}\log(-p)\right)$. The rate assumed by  $\mathfrak{D}$ is
\begin{align*} 
     \Theta \left( (-p)^{-1+\frac{1}{i_3}} \right)
     -\Theta \left( - (-p)^{-1+ \frac{1}{i}} \log (-p)\right) = \Theta \left(  (-p)^{-1+ \frac{1}{i_3}}  \right) \;.
 \end{align*}
Overall, we find the rate of the upper bound in the case $i_3 > i_2 = i_1 > 1$ to be
\begin{align*} 
    \Theta \left( -(-p)^{-1+\frac{1}{i_2}} \log(-p)\right) + \Theta \left( (-p)^{-1+ \frac{1}{i_3}} \right) = \Theta \left( (-p)^{-1+ \frac{1}{i_3}}  \right) \;.
 \end{align*} 

\item Case 3 for $\mathbf{\mathfrak{D}}$.\\
Next we consider the case $i_3 = i_2 > i_1 > 1$. We continue with equation \eqref{eq:threedim_case8} and insert $i_3 = i_2 = i$ to obtain
\begin{align}  \label{eq:threedim_case3}
   & q^{-1 + \frac{1}{3} \left( \frac{1}{i_1}+\frac{2}{i} \right)} 
    \int_{\frac{1}{ q^{-1}+1}}^1 
    \int_{\left( \frac{1}{z} - 1\right)^{\frac{1}{i_1}} q^{\frac{2}{3i_1}}}^{ q^{- \frac{1}{3i_1}}}
    \int_{\left( \frac{1}{z} - 1\right)^{\frac{1}{i}} q^{\frac{1}{3i}}y_1^{-\frac{i_1}{i}}}^{ q^{-\frac{1}{3i}}}
    \left( \frac{1}{z} - 1\right)^{\frac{1}{i}} y_1^{-\frac{i_1}{i}}y_2^{-1} \, \text{d}y_2  \, \text{d}y_1 \, \text{d}z \nonumber \\
   &= q^{-1 + \frac{1}{3} \left( \frac{1}{i_1}+\frac{2}{i} \right)} \log\left( q^{-\frac{1}{3i}} \right) 
    \int_{\frac{1}{ q^{-1}+1}}^1 
    \int_{\left( \frac{1}{z} - 1\right)^{\frac{1}{i_1}} q^{\frac{2}{3i_1}}}^{ q^{- \frac{1}{3i_1}}} 
    \left( \frac{1}{z} - 1\right)^{\frac{1}{i}} y_1^{-\frac{i_1}{i}}  \, \text{d}y_1 \, \text{d}z \nonumber \\
   &- q^{-1 + \frac{1}{3} \left( \frac{1}{i_1}+\frac{2}{i} \right)} 
    \int_{\frac{1}{ q^{-1}+1}}^1 
    \int_{\left( \frac{1}{z} - 1\right)^{\frac{1}{i_1}} q^{\frac{2}{3i_1}}}^{ q^{- \frac{1}{3i_1}}} 
    \left( \frac{1}{z} - 1\right)^{\frac{1}{i}} y_1^{-\frac{i_1}{i}} \log\left( \left( \frac{1}{z} - 1\right)^{\frac{1}{i}} q^{\frac{1}{3i}}y_1^{-\frac{i_1}{i}} \right)  \, \text{d}y_1 \, \text{d}z \\
    &=- \frac{2}{3i} \underbrace{q^{-1 + \frac{1}{3} \left( \frac{1}{i_1}+\frac{2}{i} \right)} \log\left(q \right) 
    \int_{\frac{1}{ q^{-1}+1}}^1 
    \int_{\left( \frac{1}{z} - 1\right)^{\frac{1}{i_1}} q^{\frac{2}{3i_1}}}^{ q^{- \frac{1}{3i_1}}} 
    \left( \frac{1}{z} - 1\right)^{\frac{1}{i}} y_1^{-\frac{i_1}{i}}  \, \text{d}y_1 \, \text{d}z}_{\mathfrak{D}.\text{\RomanNumeralCaps{3}}} \nonumber \\
    & - \frac{1}{i} \underbrace{q^{-1 + \frac{1}{3} \left( \frac{1}{i_1}+\frac{2}{i} \right)} 
    \int_{\frac{1}{ q^{-1}+1}}^1 
    \int_{\left( \frac{1}{z} - 1\right)^{\frac{1}{i_1}} q^{\frac{2}{3i_1}}}^{ q^{- \frac{1}{3i_1}}} 
    \left( \frac{1}{z} - 1\right)^{\frac{1}{i}} \log\left( \frac{1}{z} - 1 \right) y_1^{-\frac{i_1}{i}} \, \text{d}y_1 \, \text{d}z}_{\mathfrak{D}.\text{\RomanNumeralCaps{4}}} \nonumber \\
    & + \frac{i_1}{i} \underbrace{q^{-1 + \frac{1}{3} \left( \frac{1}{i_1}+\frac{2}{i} \right)} 
    \int_{\frac{1}{ q^{-1}+1}}^1 
    \int_{\left( \frac{1}{z} - 1\right)^{\frac{1}{i_1}} q^{\frac{2}{3i_1}}}^{ q^{- \frac{1}{3i_1}}} 
    \left( \frac{1}{z} - 1\right)^{\frac{1}{i}} y_1^{-\frac{i_1}{i}} \log\left( y_1 \right)  \, \text{d}y_1 \, \text{d}z}_{\mathfrak{D}.\text{\RomanNumeralCaps{5}}} \; . \nonumber
 \end{align}
First we evaluate $\mathfrak{D}$.\RomanNumeralCaps{3} by integrating with respect to $y_1$. Since $i>i_1$ we get that $\mathfrak{D}$.\RomanNumeralCaps{3} is equal to
\begin{align} 
    &\Big(-\frac{i_1}{i}+1\Big)^{-1} \log(q)
    \Bigg(
    q^{-1 + \frac{1}{i}} \int_{\frac{1}{ q^{-1}+1}}^1 
    \left( \frac{1}{z} - 1\right)^{\frac{1}{i}} \, \text{d}z 
    - q^{-1 + \frac{1}{i_1}} \int_{\frac{1}{ q^{-1}+1}}^1 
    \left( \frac{1}{z} - 1\right)^{\frac{1}{i_1}} \, \text{d}z
    \Bigg)  \\
    &= \Big(-\frac{i_1}{i}+1\Big)^{-1}
    \Bigg(
    \underbrace{q^{-1 + \frac{1}{i}} \log(q) 
    \int_0^{ q^{-1}} 
     \frac{z'^{\frac{1}{i}}}{(z'+1)^2} \, \text{d}z'}_{\mathfrak{D}.\text{\RomanNumeralCaps{3}}.1} 
    - \underbrace{q^{-1 + \frac{1}{i_1}} \log(q)
    \int_0^{ q^{-1}}
    \frac{z'^{\frac{1}{i_1}}}{(z'+1)^2} \, \text{d}z'}_{\mathfrak{D}.\text{\RomanNumeralCaps{3}}.2}
    \Bigg) \;. \nonumber
 \end{align}
Since $i,i_1>1$, the integrals in $\mathfrak{D}.\text{\RomanNumeralCaps{3}}.1$ and in $\mathfrak{D}.\text{\RomanNumeralCaps{3}}.2$ are uniformly bounded for $q>0$. The rate of divergence of $\mathfrak{D}.\text{\RomanNumeralCaps{3}}$ is therefore $-\Theta\left(- q^{1-\frac{1}{i}}\log(q)\right)$ .\\
Through a similar approach we note that $\mathfrak{D}.\text{\RomanNumeralCaps{4}}$ is equal to
\begin{align*} 
    \Big(-\frac{i_1}{i}+1\Big)^{-1} 
    \Bigg(
    \underbrace{ q^{-1 + \frac{1}{i}} 
    \int_0^{ q^{-1}} 
     \frac{z'^{\frac{1}{i}} \log(z)}{(z'+1)^2} \, \text{d}z'}_{\mathfrak{D}.\text{\RomanNumeralCaps{4}}.1} 
    - \underbrace{ q^{-1 + \frac{1}{i_1}} 
    \int_0^{ q^{-1}}
    \frac{z'^{\frac{1}{i_1}} \log(z)}{(z'+1)^2} \, \text{d}z'}_{\mathfrak{D}.\text{\RomanNumeralCaps{4}}.2}
    \Bigg) \, .
 \end{align*}
The rate of divergence of $\mathfrak{D}.\text{\RomanNumeralCaps{4}}$ is $\Theta\left(q^{1-\frac{1}{i}}\right)$ because the integrals in $\mathfrak{D}.\text{\RomanNumeralCaps{4}}.1$ and $\mathfrak{D}.\text{\RomanNumeralCaps{4}}.2$ are uniformly bounded for $q>0$ since $i,i_1>1$.\\
Lastly, through integration by parts we obtain that $\mathfrak{D}.\text{\RomanNumeralCaps{5}}$ is equal to
\begin{align}  \label{case3_3dim}
    &\Big(-\frac{i_1}{i}+1\Big)^{-1} 
    q^{-1 + \frac{1}{i}} \log\left(q^{-\frac{1}{3 i_1}}\right) \int_{\frac{1}{ q^{-1}+1}}^1 
    \left( \frac{1}{z} - 1\right)^{\frac{1}{i}} \, \text{d}z \nonumber \\
    & -\Big(-\frac{i_1}{i}+1\Big)^{-2} 
    q^{-1 + \frac{1}{i}}  \int_{\frac{1}{ q^{-1}+1}}^1 
    \left( \frac{1}{z} - 1\right)^{\frac{1}{i}} \, \text{d}z \nonumber \\
    &- \Big(-\frac{i_1}{i}+1\Big)^{-1} 
    q^{-1 + \frac{1}{i_1}} \int_{\frac{1}{ q^{-1}+1}}^1 
    \left( \frac{1}{z} - 1\right)^{\frac{1}{i_1}} \log\left(\left( \frac{1}{z} - 1\right)^{\frac{1}{i_1}} q^{\frac{2}{3i_1}}\right) \, \text{d}z \nonumber \\
    &+ \Big(-\frac{i_1}{i}+1\Big)^{-2}
    q^{-1 + \frac{1}{i_1}} 
    \int_{\frac{1}{ q^{-1}+1}}^1 
    \left( \frac{1}{z} - 1\right)^{\frac{1}{i_1}} \, \text{d}z\\
    &=-\Big(-\frac{i_1}{i}+1\Big)^{-2} 
    \underbrace{q^{-1 + \frac{1}{i}}  \int_0^{ q^{-1}} 
    \frac{z'^{\frac{1}{i}}}{(z'+1)^2} \, \text{d}z'}_{\mathfrak{D}.\text{\RomanNumeralCaps{5}}.1} 
    + \Big(-\frac{i_1}{i}+1\Big)^{-2} 
    \underbrace{q^{-1 + \frac{1}{i_1}} \int_0^{ q^{-1}}
    \frac{z'^{\frac{1}{i_1}}}{(z'+1)^2} \, \text{d}z'}_{\mathfrak{D}.\text{\RomanNumeralCaps{5}}.2} \nonumber \\
    &-\Big(-\frac{i_1}{i}+1\Big)^{-1} 
     \frac{1}{3i_1} \underbrace{q^{-1 + \frac{1}{i}} \log\left(q\right) 
     \int_0^{ q^{-1}} 
    \frac{z'^{\frac{1}{i}}}{(z'+1)^2} \, \text{d}z'}_{\mathfrak{D}.\text{\RomanNumeralCaps{5}}.3}
    - \Big(-\frac{i_1}{i}+1\Big)^{-1} 
    \frac{2}{3i_1} \underbrace{q^{-1 + \frac{1}{i_1}}  \log\left(q\right) 
     \int_0^{ q^{-1}} 
    \frac{z'^{\frac{1}{i_1}}}{(z'+1)^2} \, \text{d}z'}_{\mathfrak{D}.\text{\RomanNumeralCaps{5}}.4} \nonumber
    \\
    &- \Big(-\frac{i_1}{i}+1\Big)^{-1} 
    \frac{1}{i_1} \underbrace{q^{-1 + \frac{1}{i_1}} 
    \int_0^{ q^{-1}}
    \frac{z'^{\frac{1}{i_1}} \log(z')}{(z'+1)^2} \, \text{d}z'}_{\mathfrak{D}.\text{\RomanNumeralCaps{5}}.5} \;. \nonumber
 \end{align}
The rate of divergence of $\mathfrak{D}.\text{\RomanNumeralCaps{5}}$ is $\Theta\left(-q^{-1+\frac{1}{i}}\log(q)\right)$, since the integrals in \eqref{case3_3dim} are uniformly bounded for $q>0$.\\
Overall, the rate assumed by $\mathfrak{D}$ is
\begin{align*}
    \Theta\left(-q^{1-\frac{1}{i}} \log(q) \right)
    - \Theta\left(q^{1-\frac{1}{i}}\right)
    + \Theta\left(-q^{-1+\frac{1}{i}}\log(q)\right)
    = \Theta\left(-q^{-1+\frac{1}{i}}\log(q)\right)\;.
\end{align*}
We find the rate of divergence for the upper bound in the case $i_3 =i_2 > i_1 > 1$ given by
\begin{align*} 
    &\Theta \left( (-p)^{-1+ \frac{1}{i_2}} \right) 
    +  \Theta \left( -(-p)^{-1 + \frac{1}{i_3}} \log(-p) \right)
    = \Theta\left(-(-p)^{-1 + \frac{1}{i_3}} \log(-p) \right) \;.
 \end{align*}
\item Case 4 for $\mathbf{\mathfrak{D}}$.\\
We suppose now that $i_3 = i_2 = i_1 > 1$. We trace the previous case until equation \eqref{eq:threedim_case3} and insert $i = i_3 = i_2 = i_1$. Therefore $\mathfrak{D}.\text{\RomanNumeralCaps{3}}$ is equal to
\begin{align}  \label{eq:threedim_case4_1}
    &q^{-1 + \frac{1}{i}} \log\left(q \right) 
    \int_{\frac{1}{ q^{-1}+1}}^1 
    \int_{\left( \frac{1}{z} - 1\right)^{\frac{1}{i}} q^{\frac{2}{3i}}}^{ q^{- \frac{1}{3i}}} 
    \left( \frac{1}{z} - 1\right)^{\frac{1}{i}} y_1^{-1}  \, \text{d}y_1 \, \text{d}z \nonumber \\
    &= -\frac{1}{i} q^{-1 + \frac{1}{i}} \log\left(q \right) 
    \left( \log(q) 
    \int_{\frac{1}{ q^{-1}+1}}^1 
    \left( \frac{1}{z} - 1\right)^{\frac{1}{i}} \, \text{d}z 
    + \int_{\frac{1}{ q^{-1}+1}}^1 
    \left( \frac{1}{z} - 1\right)^{\frac{1}{i}} \log\left(\frac{1}{z} - 1\right) \, \text{d}z 
    \right) \\
    &= -\frac{1}{i} \underbrace{q^{-1 + \frac{1}{i}} \log^2\left(q \right) 
    \int_0^{q^{-1}}
    \frac{z'^{\frac{1}{i}}}{(z'+1)^2} \, \text{d}z'}_{\mathfrak{D}.\text{\RomanNumeralCaps{3}}.3} 
    -\frac{1}{i} \underbrace{q^{-1 + \frac{1}{i}} \log\left(q \right) \int_0^{q^{-1}} 
    \frac{z'^{\frac{1}{i}}\log(z')}{(z'+1)^2} \, \text{d}z}_{\mathfrak{D}.\text{\RomanNumeralCaps{3}}.4} \;. \nonumber
 \end{align}
Since $i>1$, the integrals in the right-hand side of \eqref{eq:threedim_case4_1} are uniformly bounded for any $q>0$. Therefore the rate of divergence of $\mathfrak{D}.\text{\RomanNumeralCaps{3}}$ is $-\Theta\left(q^{-1+\frac{1}{i}}\log^2(q)\right)$.\\
Similarly, the summand $\mathfrak{D}.\text{\RomanNumeralCaps{4}}$ assumes value
\begin{align}  \label{eq:threedim_case4_2}
    &q^{-1 + \frac{1}{i}} 
    \int_{\frac{1}{ q^{-1}+1}}^1 
    \int_{\left( \frac{1}{z} - 1\right)^{\frac{1}{i}} q^{\frac{2}{3i}}}^{ q^{- \frac{1}{3i}}} 
    \left( \frac{1}{z} - 1\right)^{\frac{1}{i}} \log\left( \frac{1}{z} - 1 \right) y_1^{-1} \, \text{d}y_1 \, \text{d}z \nonumber \\
    &= -\frac{1}{i} q^{-1 + \frac{1}{i}}
    \left( \log(q) 
    \int_{\frac{1}{ q^{-1}+1}}^1 
    \left( \frac{1}{z} - 1\right)^{\frac{1}{i}} \log\left( \frac{1}{z} - 1 \right) \, \text{d}z 
    + \int_{\frac{1}{ q^{-1}+1}}^1 
    \left( \frac{1}{z} - 1\right)^{\frac{1}{i}} \log^2\left(\frac{1}{z} - 1\right) \, \text{d}z 
    \right) \\
    &= -\frac{1}{i} \underbrace{q^{-1 + \frac{1}{i}} \log(q) 
    \int_0^{q^{-1}}
    \frac{z'^{\frac{1}{i}}\log(z')}{(z'+1)^2} \, \text{d}z'}_{\mathfrak{D}.\text{\RomanNumeralCaps{4}}.3} 
    -\frac{1}{i} \underbrace{q^{-1 + \frac{1}{i}} \int_0^{q^{-1}} 
    \frac{z'^{\frac{1}{i}}\log^2(z')}{(z'+1)^2} \, \text{d}z'}_{\mathfrak{D}.\text{\RomanNumeralCaps{4}}.4} 
    \;. \nonumber
 \end{align}
Again, since $i>1$, the integrals in the right-hand side of \eqref{eq:threedim_case4_2} are uniformly bounded for any $q>0$ and the rate of divergence of $\mathfrak{D}.\text{\RomanNumeralCaps{4}}$ is in this case $\Theta\left(-q^{-1+\frac{1}{i}}\log(q)\right)$.\\
Lastly, the summand $\mathfrak{D}.\text{\RomanNumeralCaps{5}}$ is equal to
\begin{align}  \label{eq:threedim_case4_3}
    &q^{-1 + \frac{1}{i}} 
    \int_{\frac{1}{ q^{-1}+1}}^1 
    \int_{\left( \frac{1}{z} - 1\right)^{\frac{1}{i}} q^{\frac{2}{3i}}}^{ q^{- \frac{1}{3i}}} 
    \left( \frac{1}{z} - 1\right)^{\frac{1}{i}} y_1^{-1} \log\left( y_1 \right)  \, \text{d}y_1 \, \text{d}z \nonumber \\
    & =\frac{1}{2} q^{-1 + \frac{1}{i}} \log^2(q^{-\frac{1}{3i}})
    \int_{\frac{1}{ q^{-1}+1}}^1  
    \left( \frac{1}{z} - 1\right)^{\frac{1}{i}} \, \text{d}z 
    - \frac{1}{2} q^{-1 + \frac{1}{i}}
    \int_{\frac{1}{ q^{-1}+1}}^1  
    \left( \frac{1}{z} - 1\right)^{\frac{1}{i}} 
    \log^2\left( \left( \frac{1}{z} - 1\right)^{\frac{1}{i}} q^{\frac{2}{3i}} \right) \, \text{d}z \nonumber
    \\
    &=-\frac{1}{6 i^2} q^{-1 + \frac{1}{i}} \log^2(q)
    \int_{\frac{1}{ q^{-1}+1}}^1  
    \left( \frac{1}{z} - 1\right)^{\frac{1}{i}} \, \text{d}z
    -\frac{2}{3 i^2} q^{-1 + \frac{1}{i}} \log(q)
    \int_{\frac{1}{ q^{-1}+1}}^1  
    \left( \frac{1}{z} - 1\right)^{\frac{1}{i}} \log\left( \frac{1}{z} - 1\right) \, \text{d}z \nonumber \\
    & -\frac{1}{2 i^2} q^{-1 + \frac{1}{i}}
    \int_{\frac{1}{ q^{-1}+1}}^1  
    \left( \frac{1}{z} - 1\right)^{\frac{1}{i}} \log^2\left( \frac{1}{z} - 1\right) \, \text{d}z \\
    &=-\frac{1}{6 i^2} \underbrace{q^{-1 + \frac{1}{i}} \log^2(q)
    \int_0^{q^{-1}}
    \frac{z'^{\frac{1}{i}}}{(z'+1)^2} \, \text{d}z'}_{\mathfrak{D}.\text{\RomanNumeralCaps{5}}.6}
    -\frac{2}{3 i^2} \underbrace{q^{-1 + \frac{1}{i}} \log(q)
    \int_0^{q^{-1}} 
    \frac{z'^{\frac{1}{i}}\log(z')}{(z'+1)^2} \, \text{d}z'}_{\mathfrak{D}.\text{\RomanNumeralCaps{5}}.7} \nonumber \\
    &-\frac{1}{2 i^2} \underbrace{q^{-1 + \frac{1}{i}}
    \int_0^{q^{-1}}  
    \frac{z'^{\frac{1}{i}}\log^2(z')}{(z'+1)^2} \, \text{d}z'}_{\mathfrak{D}.\text{\RomanNumeralCaps{5}}.8} \;. \nonumber
 \end{align}
Since $i>1$, the integrals in the right-hand side of \eqref{eq:threedim_case4_3} are uniformly bounded for any $q>0$ and the rate of divergence of $\mathfrak{D}.\text{\RomanNumeralCaps{5}}$ is in this case $-\Theta\left(q^{-1+\frac{1}{i}}\log^2(q)\right)$.\\
Observing \eqref{eq:threedim_case3}, \eqref{eq:threedim_case4_1} and \eqref{eq:threedim_case4_3}, we can note that $\mathfrak{D}$ has as leading terms $\frac{2}{3 i^2} \mathfrak{D}.\text{\RomanNumeralCaps{3}}.3$ and $-\frac{1}{6 i^2}\mathfrak{D}.\text{\RomanNumeralCaps{5}}.6$, hence it assumes the same rate of divergence as
\begin{align*}
    \frac{2}{3 i^2} \mathfrak{D}.\text{\RomanNumeralCaps{3}}.3
    -\frac{1}{6 i^2}\mathfrak{D}.\text{\RomanNumeralCaps{5}}.6
    =\frac{1}{2 i^2} q^{-1 + \frac{1}{i}} \log^2\left(q \right) 
    \int_0^{q^{-1}}
    \frac{z'^{\frac{1}{i}}}{(z'+1)^2} \, \text{d}z'=\Theta\left(q^{-1 + \frac{1}{i}} \log^2\left(q \right)  \right) \;.
\end{align*}
Combining both summands $\mathfrak{C}$ and $\mathfrak{D}$, we obtain the rate of the upper bound as 
\begin{align*} 
   \Theta \left( -(-p)^{-1+ \frac{1}{i_2}} \log (-p) \right) + \Theta \left( (-p)^{-1 + \frac{1}{i}} \log^2 (-p) \right) = \Theta \left( (-p)^{-1+ \frac{1}{i}} \log^2 (-p) \right) 
 \end{align*}
for $i_3 = i_2 = i_1 > 1$.
\item Case 5 for $\mathbf{\mathfrak{D}}$.\\
The next case we consider is $i_3 > i_2 > i_1 = 1$. We can follow the same steps as in the first case for $\mathfrak{D}$ but we note that the integrals in $\mathfrak{D}.\text{\RomanNumeralCaps{1}}.2$ and $\mathfrak{D}.\text{\RomanNumeralCaps{2}}.2$ diverge as $\Theta\left(\log\left(q^{-1}\right)\right)$, due to \eqref{integral_tool}. The rate of divergence of $\mathfrak{D}$ is therefore
\begin{align*} 
    \Theta\left( q^{-1 + \frac{1}{i_3}} \right)
    - \Theta\left( -\log\left(q\right)\right) - \Theta\left( q^{-1 + \frac{1}{i_2}}\right) + \Theta\left( -\log\left(q\right)\right)  = \Theta\left( q^{-1 + \frac{1}{i_3}} \right)
 \end{align*}
and the rate of the upper bound in the case $i_3 > i_2 > i_1 = 1$ is
\begin{align*} 
     \Theta\left( (-p)^{-1+ \frac{1}{i_2}} \right) + \Theta\left( (-p)^{-1 + \frac{1}{i_3}} \right) = \Theta\left( (-p)^{-1 + \frac{1}{i_3}} \right) \,.
 \end{align*}

\item Case 6 for $\mathbf{\mathfrak{D}}$.\\
We assume that $i_3 > i_2 = i_1 = 1$ and proceed as in the previous case as we prove that the rate of divergence of $\mathfrak{D}.\text{\RomanNumeralCaps{1}}$ is 
\begin{align*}
    \Theta\left( q^{-1 + \frac{1}{i_3}} \right)
    - \Theta\left( -\log(q)\right)= \Theta\left( q^{-1 + \frac{1}{i_3}} \right) \; . 
\end{align*}
We can then follow same approach as in the second case  for $\mathfrak{D}$ and study the summands in \eqref{case2_3dim}. Inserting $i_2 = i_1 = i = 1$ we know from \eqref{integral_tool} that the rate of divergence assumed by $\mathfrak{D}.\text{\RomanNumeralCaps{2}}.3$ is $-\Theta\left(\log^2(q) \right)$ and from \eqref{eq:integral3}, in Appendix \ref{AppendC}, that the rate of $\mathfrak{D}.\text{\RomanNumeralCaps{2}}.4$ is $\Theta\left( \log^2(q) \right)$. Hence, we can state that $\mathfrak{D}=\Theta\left( q^{-1 + \frac{1}{i_3}} \right)$, and in the case $i_3 > i_2 = i_1 = 1$ that the rate of divergence of the upper bound is
\begin{align*} 
     \Theta \left( \log^2 (-p)\right) + \Theta \left( (-p)^{-1+ \frac{1}{i_3}} \right) = \Theta \left(  (-p)^{-1+ \frac{1}{i_3}} \right) \, .
 \end{align*}

\item Case 7 for $\mathbf{\mathfrak{D}}$.\\
We now consider $i_3 = i_2 > i_1 = 1$. We follow the same approach taken in the third case for $\mathfrak{D}$ and set $i_3=i_2=i$. We obtain that, due to \eqref{integral_tool} and \eqref{eq:integral3},
\begin{alignat*}{4}
    &\mathfrak{D}.\text{\RomanNumeralCaps{3}}.1=-\Theta\left(
    -q^{1-\frac{1}{i}} \log\left(q\right)\right) \quad 
    &&, \quad
    \mathfrak{D}.\text{\RomanNumeralCaps{3}}.2=-\Theta\left(\log^2(q) \right) \quad 
    &&, \quad
    \mathfrak{D}.\text{\RomanNumeralCaps{4}}.1=\Theta\left( q^{1-\frac{1}{i}} \right) \quad 
    &&, \\
    &\mathfrak{D}.\text{\RomanNumeralCaps{4}}.2=\Theta\left(
    \log^2(q) \right) \quad 
    &&, \quad
    \mathfrak{D}.\text{\RomanNumeralCaps{5}}.1=\Theta\left(
    q^{1-\frac{1}{i}} \right) \quad
    &&, \quad
    \mathfrak{D}.\text{\RomanNumeralCaps{5}}.2=\Theta\left( -\log\left(q\right) \right) \quad 
    &&, \\
    &\mathfrak{D}.\text{\RomanNumeralCaps{5}}.3=-\Theta\left(
    -q^{1-\frac{1}{i}} \log\left(q\right) \right) \quad 
    &&, \quad
    \mathfrak{D}.\text{\RomanNumeralCaps{5}}.4=-\Theta\left(
    \log^2\left(q\right) \right) \quad 
    &&, \quad
    \mathfrak{D}.\text{\RomanNumeralCaps{5}}.5=\Theta\left(
    \log^2\left(q\right) \right) \quad 
    &&.
\end{alignat*}
The rate of divergence of $\mathfrak{D}$ is hence
\begin{align*} 
    &\Theta\left(- q^{1-\frac{1}{i}} \log(q)\right)
    -\Theta\left(\log^2(q) \right)
    -\Theta\left( q^{1-\frac{1}{i}} \right)
    +\Theta\left(\log^2(q) \right)
    -\Theta\left(q^{1-\frac{1}{i}} \right)\\
    &+\Theta\left( -\log\left(q\right) \right)
    +\Theta\left(-q^{1-\frac{1}{i}} \log\left(q\right)\right)
    +\Theta\left(\log^2\left(q\right) \right)
    -\Theta\left(\log^2\left(q\right) \right)
    =\Theta\left(-q^{1-\frac{1}{i}} \log\left(q\right)\right)\;.
 \end{align*}
Overall, combining $\mathfrak{C}$ and $\mathfrak{D}$ we get the rate of divergence of upper bound as
\begin{align*} 
    &\Theta \left( (-p)^{-1 + \frac{1}{i_2}} \right) + \Theta \left(- (-p)^{-1+ \frac{1}{i_3}} \log(-p) \right) 
    = \Theta \left( -(-p)^{-1+ \frac{1}{i_3}} \log(-p) \right) 
 \end{align*}
for $i_3 = i_2 > i_1 = 1$.

\item Case 8 for $\mathbf{\mathfrak{D}}$.\\
The last case for which we evaluate $\mathfrak{D}$ is $i_3 = i_2 = i_1 = 1$. Following the same steps of the fourth case for $\mathfrak{D}$ we obtain \eqref{eq:threedim_case4_1}, \eqref{eq:threedim_case4_2} and \eqref{eq:threedim_case4_3}. We know from \eqref{integral_tool} and \eqref{nice2}, in Appendix \ref{AppendC}, that $\mathfrak{D}.\text{\RomanNumeralCaps{3}}.3$, 
$\mathfrak{D}.\text{\RomanNumeralCaps{4}}.4$, 
$\mathfrak{D}.\text{\RomanNumeralCaps{5}}.6$, and
$\mathfrak{D}.\text{\RomanNumeralCaps{5}}.8$ diverge with rate $\Theta\left(- \log^3(q)\right)$ and that $\mathfrak{D}.\text{\RomanNumeralCaps{3}}.4$, $\mathfrak{D}.\text{\RomanNumeralCaps{4}}.3$, $\mathfrak{D}.\text{\RomanNumeralCaps{5}}.7$ as $-\Theta\left(- \log^3(q)\right)$. \\
From \eqref{integral_tool}, \eqref{nice} and \eqref{nice2}, we get the rate of divergence of the summand $\mathfrak{D}$ as
\begin{align*} 
    &\frac{2}{3} \mathfrak{D}.\text{\RomanNumeralCaps{3}}.3 + 
    \frac{2}{3} \mathfrak{D}.\text{\RomanNumeralCaps{3}}.4 +
    \mathfrak{D}.\text{\RomanNumeralCaps{4}}.3 +
    \mathfrak{D}.\text{\RomanNumeralCaps{4}}.4 -
    \frac{1}{6} \mathfrak{D}.\text{\RomanNumeralCaps{5}}.6 -
    \frac{2}{3} \mathfrak{D}.\text{\RomanNumeralCaps{5}}.7 -
    \frac{1}{2} \mathfrak{D}.\text{\RomanNumeralCaps{5}}.8\\
    &=\left(\frac{2}{3}-\frac{2}{3} \frac{1}{2} -\frac{1}{2}+\frac{1}{3} -\frac{1}{6}+\frac{2}{3} \frac{1}{2} -\frac{1}{2} \frac{1}{3} \right) \Theta\left(- \log^3(q)\right)
    =\Theta\left(-\log^3(q)\right)
 \end{align*}
Conclusively, in the case $i_3 = i_2 = i_1 = 1$ we find that $\langle V_{\infty}g, g\rangle$ has an upper bound with rate of divergence
\begin{align*} 
     \Theta \left( \log^2(-p)\right) + \Theta \left( -\log^3 (-p) \right)  = \Theta \left( -\log^3 (-p) \right)  \, .
 \end{align*}
\end{itemize}

\end{proof}

\section{Appendix} \label{AppendC}

For $q>0$ and $m\in\mathbb{N}\cup\{0\}$
\begin{align}  \label{eq:integral3}
    \int_0^{ q^{-1}}\frac{z^{\prime }\log^m \left( z^{\prime} \right)}{(z^{\prime}+1)^2} \, \text{d}z^{\prime} = \Theta\left( \left(-\log(q)\right)^{m+1}\right)  
 \end{align}
holds true by considering
\begin{align*} 
    \lim_{q \to 0^{+}}\int_{0}^{ q^{-1}} \frac{z^{\prime}\log^m \left(z^{\prime}\right)}{(z^{\prime}+1)^2} \, \text{d}z^{\prime} &= \int_{0}^{n} \frac{z^{\prime}\log^m \left(z^{\prime}\right)}{(z^{\prime}+1)^2} \, \text{d}z^{\prime} + \lim_{q \to 0^{+}}\int_{n}^{q^{-1}} \frac{z^{\prime}\log^m \left(z^{\prime}\right)}{(z^{\prime}+1)^2} \, \text{d}z^{\prime} \;,
 \end{align*}
with $n>0$. Whereas the first integral is finite for any finite $n>0$, for the second we find that, in the limit $z^{\prime}\to +\infty$, $\frac{z^{\prime}\log^m \left(z^{\prime} \right)}{(z^{\prime}+1)^2}$ behaves equivalently to $\frac{\log^m\left(z^{\prime}\right)}{z^{\prime}}$ since 
\begin{align}  \label{nice}
    \lim_{z^{\prime} \to \infty} \frac{(z^{\prime}+1)^{-2}z^{\prime}\log^m \left(z^{\prime}\right)}{ z^{\prime -1} \log^m \left( z^{\prime} \right)} 
     = 1 \quad .
 \end{align}
We study then the integral
\begin{align*} 
    &\int_{n}^{ q^{-1}} \frac{\log^m \left( z^{\prime} \right)}{z^{\prime}} \, \text{d}z^{\prime} 
    =  \left[\frac{1}{m+1}\log^{m+1}(z^{\prime})\right]_n^{ q^{-1}}
    =  \frac{1}{m+1}\log^{m+1}( q^{-1})-\frac{1}{m+1}\log^{m+1}(n)
\end{align*}
which concludes the proof.\\
Lastly we note that
\begin{align} \label{nice2}
   \lim_{q \to 0^+} \frac{1}{\frac{1}{m+1}\log^{m+1}( q^{-1})}
   \int_0^{ q^{-1}}\frac{z^{\prime }\log^m \left( z^{\prime} \right)}{(z^{\prime}+1)^2} \, \text{d}z^{\prime}
   =1 \quad .
\end{align}

\newpage
\printbibliography

\end{document}